\documentclass[12pt,a4paper,article] {amsart}
\usepackage{amssymb}
\usepackage{amscd}

\usepackage{geometry}
\usepackage[leqno]{amsmath}
\usepackage{amsfonts}
\usepackage{amssymb}
\usepackage{amsthm}
\usepackage{amssymb}
\usepackage[all]{xy}
\usepackage{graphicx}
\usepackage{hyperref}
\usepackage{color}

\newcommand{\CCChi}{\mbox{\Huge$\chi$}}

\newcommand{\RR}{\mathbf R}
\newcommand{\NN}{\mathbf N}

\newcommand{\ZZ}{\mathbf Z}

\DeclareMathOperator{\esssup}{ess\,sup}
\DeclareMathOperator{\essinf}{ess\,inf}

\theoremstyle{plain}
\newtheorem{theorem}{Theorem}[section]
\newtheorem{corollary}[theorem]{Corollary}
\newtheorem{lemma}[theorem]{Lemma}

\theoremstyle{definition}
\newtheorem{definition}[theorem]{Definition}

\newtheorem{remark}[theorem]{Remark}
\newtheorem{notation}[theorem]{Notation}

\numberwithin{equation}{section}

\renewcommand{\leq}{\leqslant}
\renewcommand{\geq}{\geqslant}

\title[Random homogenization of coercive Hamilton-Jacobi equations in 1d]{Random homogenization of coercive Hamilton-Jacobi equations in 1d}



\author[Hongwei Gao]{Hongwei Gao}

\address{
University of California, Irvine   \\ 
CA, 92697\\
USA}
\email{hongweig@math.uci.edu}

\thanks{Partially supported by DMS-1151919 } 

\begin{document}
\vspace{18mm} \setcounter{page}{1} \thispagestyle{empty}

\begin{abstract}
In this paper, we will prove the random homogenization of general coercive non-convex Hamilton-Jacobi equations in one dimensional case. This extends the result of Armstrong, Tran and Yu when the Hamiltonian has a separable form $H(p,x,\omega)=H(p)+V(x,\omega)$ for any coercive $H(p)$.
\end{abstract}

\maketitle

\section{Introduction}

\subsection{Overview}

We study the Hamilton-Jacobi equation of the following form:
\begin{eqnarray*}
\begin{cases}
u_{t}+H(Du,x,\omega)=0 & (x,t)\in\RR^{d}\times(0,+\infty)\\
u(x,0)= g(x) & x\in\RR^{d}\\
\end{cases}
\end{eqnarray*}

The Hamiltonian $H(p,x,\omega)$ is stationary ergodic and $g(x)\in BUC(\RR^{d})$. The main issue in stochastic homogenization of Hamilton-Jacobi equation is to consider: for each $\epsilon>0$, $\omega\in\Omega$, let $u^{\epsilon}(t,x,\omega)$ be the unique solution of the equation:
\begin{eqnarray*}
\begin{cases}
u^{\epsilon}_{t}+H(Du^{\epsilon},\frac{x}{\epsilon},\omega)=0 & (x,t)\in\RR^{d}\times(0,+\infty)\\
u^{\epsilon}(x,0)= g(x) & x\in\RR^{d}\\
\end{cases}
\end{eqnarray*}

Prove that for a.e. $\omega\in\Omega$, as $\epsilon\rightarrow 0$, $u^{\epsilon}(t,x,\omega)\rightarrow \overline{u}(t,x)$ locally uniformly and $\overline{u}(t,x)$ is the unique solution of the homogenized equation:
\begin{eqnarray*}
\begin{cases}
\overline{u}_{t}+\overline{H}(D\overline{u})=0 & (x,t)\in\RR^{d}\times(0,+\infty)\\
\overline{u}(x,0)=g(x) & x\in\RR^{d}\\
\end{cases}
\end{eqnarray*}

If $H(p,x,\omega)$ is convex with respect to $p\in\RR^{d}$, stochastic homogenization was proved independently by Souganidis[\ref{99 Souganidis}] and by Rezakhanlou-Tarver[\ref{00Rezakhanlou, Tarver}]. This result was extended to t-dependent case by Schwab [\ref{russell schwab}] when the Hamiltonian has super-linear growth in $p$ and by Jing-Souganidis-Tran[\ref{Jing-Souganidis-Tran}] for Hamiltonians with the form $a(x,t,\omega)|p|$. For those quasi-convex Hamiltonians, Siconolfi and Davini [\ref{Davini and Siconolfi}] established the random homogenization in 1d, and the general dimensional case was proved by Amstrong-Souganidis[\ref{ref Armstrong and Souganidis}].

It remains an open problem that whether random homogenization still holds if the Hamiltonian is non-convex. The first genuinely non-convex example of stochastic homogenization was provided by Amstrong-Tran-Yu[\ref{a non-covex by ATY}] for a special class of  Hamiltonians with the following typical form.
\begin{eqnarray*}
H(p,x,\omega)=(|p|^{2}-1)^{2}+V(x,\omega), (p,x)\in\RR^{d}\times\RR^{d}
\end{eqnarray*}

In one dimensional case, the same author established in another paper[\ref{1-d seperable noncovex by ATY}] the random homogenization of separable Hamiltonians
\begin{eqnarray*}
H(p,x,\omega)=H(p)+V(x,\omega),(p,x)\in\RR\times\RR  & \text{for any coercive }H(p)
\end{eqnarray*}

Recently,  Armstrong-Cardaliaguet[\ref{Armstrong and Cardaliaguet}] considered the homogenization of Hamiltonian $H(p,x,\omega)$ that is homogeneous in $p$ and with the assumption of unit range of dependence on $(x,\omega)$(basically, it means that $H(p,x,\omega)$ and $H(p,y,\omega)$ are independent once $|x-y|>1$).

This paper is aimed to extend the result of Amstrong-Tran-Yu[\ref{1-d seperable noncovex by ATY}] to general coercive $H(p,x,\omega)$.

\subsection{Assumption and main result}

Consider the Hamiltonian $H(p,x,\omega)$ that is continuous in $(p,x)\in\RR\times\RR$ and measurable in $\omega\in\Omega$.

\textbf{(A1)} Stationary Ergodic: there exists a probability space $\ (\Omega,\mathcal{F},\mathbf{P})$ and a group $\lbrace\tau_{y}\rbrace_{y\in\RR}$ of $\mathcal{F}$-measurable, measure-preserving transformations $\tau_{y}:\Omega \rightarrow \Omega$, i.e. for any $x,y\in\RR$:
\begin{eqnarray*}
\tau_{x+y}=\tau_{x}\circ \tau_{y} \text{ and } \mathbf{P}[\tau_{y}(A)]=\mathbf{P}[A]
\end{eqnarray*}

\hspace{1cm}Ergodic: $A\in \mathcal{F}, \ \tau_{z}(A)=A \text{ for every }z\in \RR \Rightarrow \mathbf{P}[A]\in \lbrace 0,1 \rbrace$.

\hspace{1cm}Stationary: $H(p,y,\tau_{z}\omega)=H(p,y+z,\omega) \text{ for any }y,z\in\RR \text{ and }\omega\in\Omega$.

\vspace{5mm}

\textbf{(A2)} Coercive: $\liminf\limits_{|p|\rightarrow +\infty}\essinf\limits\limits_{(x,\omega)\in\RR\times\Omega} H(p,x,\omega)=+\infty$.

\vspace{5mm}

\textbf{(A3)} Local Uniformly Continuous: for any compact set $K\subset\RR$,
\begin{eqnarray*}
|H(p,x,\omega)-H(q,y,\omega)|\leq \rho_{K}(|p-q|+|x-y|), (p,x,\omega),(q,y,\omega)\in K\times\RR\times\Omega
\end{eqnarray*}

The above $\rho_{K}$ is the modulus of continuity.

\begin{theorem}\label{main theorem}
Assume \textbf{(A1)-(A3)} hold and $g(x)\in BUC(\RR)$, for each $\epsilon>0$ and $\omega\in\Omega$, let $u^{\epsilon}(x,t,\omega)$ be the solution of the Hamilton-Jacobi equation
\begin{eqnarray*}
\begin{cases}
u_{t}^{\epsilon}+H(Du^{\epsilon},\frac{x}{\epsilon},\omega)=0 & (x,t)\in\RR\times(0,+\infty)\\
u^{\epsilon}(x,0)=g(x) & x\in\RR\\
\end{cases}
\end{eqnarray*}
Then, there is an effective Hamiltonian $\overline{H}(p)\in C(\RR)$ with $\lim\limits_{|p|\rightarrow+\infty}\overline{H}(p)=+\infty$, such that for a.e. $\omega\in\Omega$, $\lim\limits_{\epsilon\rightarrow 0^{+}}u^{\epsilon}(x,t,\omega)=\overline{u}(x,t)$ locally uniformly and $\overline{u}(x,t)$ is the solution of the homogenized Hamilton-Jacobi equation
\begin{eqnarray*}
\begin{cases}
\overline{u}_{t}+\overline{H}(D\overline{u})=0 & (x,t)\in\RR\times(0,+\infty)\\
\overline{u}(x,0)=g(x) & x\in\RR\\
\end{cases}
\end{eqnarray*}

\end{theorem}



\subsection{Main difficulty and main idea}
Let's first review the case of separable Hamiltonian, by approximation, we can assume $H(p)$ has finite many wells. The main ingredients in the proof by Amstrong-Tran-Yu[\ref{1-d seperable noncovex by ATY}] are the following: (1)When the oscillation of $V(x,\omega)$ is larger than the maximal local oscillation of $H(p)$, $\overline{H}(p)$ turns out to be quasi-convex. (2) If $V(x,\omega)$ has small oscillation, they introduced gluing lemmas, through which the Hamiltonian can be eventually reduced to the large oscillation case.

For the general Hamiltonian $H(p,x,\omega)$, there are several difficulties we need to overcome. 

First, unlike the separable Hamiltonian, the number of wells of $H(p,x,\omega)$ (as a function of $p$) depends on $(x,\omega)$. To solve this problem, we approximate $H(p,x,\omega)$ by Hamiltonians that have same number of wells for every $(x,\omega)$(c.f. section \ref{reduction to constrained Hamiltonian}). 

Secondly, we need to find a way to characterize the oscillation when $p$ and $(x,\omega)$ are mixed. After that, we can extend the above (1) and (2) to our general situation.

\section{Preliminaries}\label{preliminary}

\subsection{Stability of Homogenization}

\begin{definition}(Follow the definition in Armstrong-Tran-Yu[\ref{1-d seperable noncovex by ATY}])
$H(p,x,\omega)$ is regularly homogenizable at $p\in\RR$ if there exists an $\overline{H}(p)\in\RR$ such that: for any $\lambda>0$, if $v_{\lambda}(x,p,\omega)\in W^{1,\infty}(\RR)$ is the unique viscosity solution of the equation
\begin{eqnarray*}
	\lambda v_{\lambda}+H(v_{\lambda}^{\prime},x,\omega)=0 && x\in\RR
\end{eqnarray*}
Then
\begin{eqnarray}\label{regularly homogenizable}
\text{for any } R>0,&&\mathbf{P}\left[\omega\in\Omega:\limsup\limits_{\lambda\rightarrow 0}\max\limits_{|x|\leq\frac{R}{\lambda}}\left|\lambda v_{\lambda}(x,p,\omega)+\overline{H}(p)\right|=0\right]=1
\end{eqnarray}

\end{definition}

\begin{remark}
By Armstrong-Souganidis[\ref{ref Armstrong and Souganidis}], with \textbf{(A1)}, (\ref{regularly homogenizable}) is equivalent to the following identity.
\begin{eqnarray*}
\mathbf{P}\left[\omega\in\Omega:\lim\limits_{\lambda\rightarrow 0}\left|\lambda v_{\lambda}(0,p,\omega)+\overline{H}(p)\right|=0\right]=1
\end{eqnarray*}
\end{remark}
\begin{remark}
Homogenization with $H(p,x,\omega)$ holds if $H(p,x,\omega)$ is regularly homogenizable for each $p\in\RR$. If the cell problem at $p$ is solvable, then $H(p,x,\omega)$ is regularly homogenizable at $p$.
\end{remark}

\begin{definition}
Let $G(p,x,\omega):\RR\times\RR\times\Omega\rightarrow\RR$ satisfy \textbf{(A1)}, denote
\begin{eqnarray*}
G_{\inf}(p):=\essinf\limits\limits_{(x,\omega)\in\RR\times\Omega}G(p,x,\omega) && G_{\sup}(p):=\esssup\limits\limits_{(x,\omega)\in\RR\times\Omega}G(p,x,\omega)
\end{eqnarray*}
\end{definition}

\begin{lemma}\label{for a.e. omega}
If $G(p,x,\omega)$ satisfies \textbf{(A1)} and is continuous in $x$,  $G_{\inf}(p),G_{\sup}(p)\in\RR$, then for a.e. $\omega\in\Omega$,
\begin{eqnarray*}
G_{\inf}(p)=\essinf\limits\limits_{x\in\RR}G(p,x,\omega) && G_{\sup}(p)=\esssup\limits\limits_{x\in\RR}G(p,x,\omega)
\end{eqnarray*}
\end{lemma}

\begin{proof}
Fix $p\in\RR$, denote $g(x,\omega)=G(p,x,\omega)$. For any $\alpha\in\RR$, define
\begin{eqnarray}\label{A_alpha}
A_{\alpha}:=\lbrace \omega\in\Omega:g(x,\omega)>\alpha \text{ for all }x\in\RR\rbrace
\end{eqnarray}

Stationary implies $\tau_{z}A_{\alpha}=A_{\alpha}$, for any $z\in\RR$. By ergodicity, $\mathbf{P}[A_{\alpha}]=0$ or 1. Let $\alpha_{0}:=\sup\lbrace \alpha: \mathbf{P}[A_{\alpha}]=1 \rbrace$. By (\ref{A_alpha}), $\alpha_{0}=G_{\inf}(p)$. Since $\mathbf{P}\left[A_{\alpha_{0}-\frac{1}{n}}\right]=1, n\in\NN$, $\mathbf{P}\left[\bigcap\limits_{n=1}^{\infty} A_{\alpha_{0}-\frac{1}{n}}\right]=1$. Then
\begin{eqnarray*}
G_{\inf}(p)=\essinf\limits\limits_{x\in\RR}G(p,x,\omega) & \omega\in \bigcap\limits_{n=1}^{\infty} A_{\alpha_{0}-\frac{1}{n}}
\end{eqnarray*}

The other equality can be proved similarly.

\end{proof}

\begin{lemma}\label{Stability}
Given uniformly coercive Hamiltonians $\lbrace H_{n}(p,x,\omega)\rbrace_{n\geq 1}\bigcup\lbrace H(p,x,\omega)\rbrace$ that satisfy \textbf{(A1)}, each $H_{n}(p,x,\omega)$ has effective Hamiltonian $\overline{H}_{n}(p)$. And for a.e. $\omega\in\Omega$:
\begin{eqnarray*}
\lim_{n\rightarrow +\infty}\lVert H_{n}(p,x,\omega)-H(p,x,\omega) \rVert_{L^{\infty}(K\times\RR)}=0 & \text{ for each compact set }K\subset\RR
\end{eqnarray*}
Then, $H(p,x,\omega)$ has effective Hamiltonian $\overline{H}(p)$ and $\lim\limits_{n\rightarrow+\infty}\overline{H}_{n}(p)=\overline{H}(p)$.
\end{lemma}

\begin{proof}
Fix $p\in\RR$, for $\lambda>0$, let $v_{n,\lambda}(x,p,\omega)$ and $v_{\lambda}(p,x,\omega)$ be solutions of the following equations:
\begin{eqnarray*}
\lambda v_{n,\lambda}+H_{n}(p+v_{n,\lambda}^{\prime},x,\omega)= 0  &&
\lambda v_{\lambda}+H(p+v_{\lambda}^{\prime},x,\omega)= 0
\end{eqnarray*}

Then
\begin{eqnarray*}
-\lambda v_{n,\lambda}\in [H_{n,\inf}(p),H_{n,\sup}(p)] &&
-\lambda v_{\lambda}\in[H_{\inf}(p),H_{\sup}(p)] 
\end{eqnarray*}

By uniform coercive, there is $r=r(p)$, such that $|v_{n,\lambda}^{\prime}(x,\omega)|,|v_{\lambda}^{\prime}(x,\omega)|< r$.

Denote $K:=[p-r,p+r]$, then by comparison principle,
\begin{eqnarray*}
|\lambda v_{n,\lambda}(0,p,\omega)-\lambda v_{\lambda}(0,p,\omega)|&\leq& \sup_{x\in\RR}|\lambda v_{n,\lambda}(x,p,\omega)-\lambda v_{\lambda}(x,p,\omega)|\\
&\leq& \lVert H_{n}(\cdot,\cdot,\omega)-H(\cdot,\cdot,\omega) \rVert_{L^{\infty}(K\times\RR)}
\end{eqnarray*}

Boundedness of $-\lambda v_{n,\lambda}$ implies $\left\lbrace \overline{H_{n}}(p)\right\rbrace_{n\geq 1}$ is bounded. For any subsequence $\lbrace n_{j} \rbrace_{j\geq 1}$, there is a sub-subsequence $\lbrace n_{j_{k}} \rbrace_{k\geq 1}$, such that $\lim\limits_{k\rightarrow \infty} \overline{H_{n_{j_{k}}}}(p_{0})=h_{*}$. Then
\begin{eqnarray*}
\left|\left(-\lambda v_{\lambda}(0,p,\omega)\right)-h_{*}\right|&\leq&  \left|\left(-\lambda v_{\lambda}(0,p,\omega)\right)-\left(-\lambda v_{n_{j_{k}},\lambda}(0,p,\omega)\right) \right|\\
&& +\left|\left(-\lambda v_{n_{j_{k}},\lambda}(0,p,\omega)\right)-\overline{H_{n_{j_{k}}}}(p)\right|+\left|\overline{H_{n_{j_{k}}}}(p)-h_{*}\right|\\
&\leq&\left\lVert H_{n_{j_{k}}}(\cdot,\cdot,\omega)-H(\cdot,\cdot,\omega) \right\rVert_{L^{\infty}\big(K\times\RR\big)}\\
& &+\left|\left(-\lambda v_{n_{j_{k}},\lambda}(0,p,\omega)\right)-\overline{H_{n_{j_{k}}}}(p)\right|+\left|\overline{H_{n_{j_{k}}}}(p)-h_{*}\right|\\
&=:&  \text{\textcircled{1}}+\text{\textcircled{2}}+\text{\textcircled{3}}
\end{eqnarray*}

For any $\epsilon>0$, when $k$ is large enough, $\text{\textcircled{1}}<\frac{\epsilon}{3}$, and $\text{\textcircled{3}}<\frac{\epsilon}{3}$. Fix such $k$, there is some $\lambda_{0}=\lambda_{0}(k)$, such that, $\text{\textcircled{2}}<\frac{\epsilon}{3}$ as long as $0<\lambda<\lambda_{0}$. Thus
$\lim\limits_{\lambda\rightarrow 0}-\lambda v_{\lambda}(0,p,\omega)=h_{*}$, so $\overline{H}(p)=h_{*}.$ The above limit is independent of the choice of $\lbrace n_{j} \rbrace_{j\geq 1}$, then $\lim\limits_{n\rightarrow\infty}\overline{H_{n}}(p)=h_{*}$. Thus
\begin{eqnarray*}
\lim_{n\rightarrow\infty}\overline{H_{n}}(p)=\overline{H}(p)
\end{eqnarray*}

\end{proof}

\begin{remark}
	Based on this lemma, we can construct the approximation of $H(p,x,\omega)$ by constrained Hamiltonians(c.f. Definition \ref{constrained Hamiltonian}), this is the first step of reduction in this paper.
\end{remark}

\begin{corollary}\label{closedness of regularly homogenizable points}
Let $H(p,x,\omega)$ satisfy \textbf{(A1)-(A3)} and fix $p_{0}\in\RR$.

(1)If $H(p,x,\omega)$ is regularly homogenizable on $(-\infty,p_{0})$ and $\overline{H}(p)$ is continuous, then $H(p,x,\omega)$ is also homogenizable at $p_{0}$ and $\lim\limits_{p\rightarrow p_{0}^{-}}\overline{H}(p)=\overline{H}(p_{0})$.

(2)If $H(p,x,\omega)$ is regularly homogenizable on $(p_{0},+\infty)$ and $\overline{H}(p)$ is continuous, then $H(p,x,\omega)$ is also homogenizable at $p_{0}$ and $\lim\limits_{p\rightarrow p_{0}^{+}}\overline{H}(p)=\overline{H}(p_{0})$.

\end{corollary}

\begin{proof}
Only prove $(1)$, since the proof of $(2)$ is similar. For any $\delta_{n}\rightarrow 0^{+}$, denote
\begin{eqnarray*}
H_{n}(p,x,\omega):=H\left(p-\delta_{n},x,\omega\right)
\end{eqnarray*}

By assumption, $H_{n}(p,x,\omega)$ is regularly homogenizable at $p_{0}$.  According to \textbf{(A3)}, for each $\omega\in\Omega$ and compact set $K\subset\RR$, we have $\lim\limits_{n\rightarrow \infty}\lVert H_{n}(p,x,\omega)-H(p,x,\omega) \rVert_{L^{\infty}(K\times\RR)}=0$.

Lemma \ref{Stability} implies $H(p,x,\omega)$ is regularly homogenizable at $p_{0}$ and 
\begin{eqnarray*}
\overline{H}(p_{0})=\lim_{n\rightarrow +\infty}\overline{H_{n}}(p_{0},x,\omega)=\lim_{n\rightarrow +\infty}\overline{H}\left(p_{0}-\delta_{n},x,\omega\right)
\end{eqnarray*}

This is true for any sequence $\delta_{n}\rightarrow 0^{+}$, so
\begin{eqnarray*}
\lim_{p\rightarrow p_{0}^{-}}\overline{H}(p,x,\omega)=\overline{H}(p_{0},x,\omega)
\end{eqnarray*}
\end{proof}

\subsection{Comparison Principle}
\begin{lemma}\label{CP}
	Let $H(p,x,\omega)$ satisfy $\textbf{(A1)-(A3)}$, for $R>0$, $1\gg\lambda>0$, $p\in\RR$, let $u$ and $v$ both be viscosity solutions of the equation
	\begin{eqnarray*}
		\lambda \gamma+H(p+\gamma^{\prime},x,\omega)=0 & x\in \left[-\frac{R}{\lambda},\frac{R}{\lambda}\right]
	\end{eqnarray*}
If there is a constant $M=M(p)>0$, such that $|\lambda u|+|\lambda v|\leq M(p)$. Then, there is a constant $C=C(p)>0$, such that
	\begin{eqnarray*}
		|\lambda u-\lambda v|\leq \frac{M(p)}{R}\sqrt{|x|^{2}+1}+\frac{M(p)C(p)}{R} & x\in \left[-\frac{R}{\lambda},\frac{R}{\lambda}\right]
	\end{eqnarray*}
\end{lemma}

\begin{proof}
	
	By $|\lambda u|+|\lambda v|\leq M(p)$, $H(p+u^{\prime},x,\omega)\leq M(p)$,
	$H(p+v^{\prime},x,\omega)\leq M(p)$.

	By \textbf{(A2)}, there is some $r=r(p)>0$, s.t. $|u^{\prime}|, |v^{\prime}|\leq r(p)$.

	By \textbf{(A3)}, there is some $\delta=\delta(p)>0$, s.t. $|H(q_{1},x,\omega)-H(q_{2},x,\omega)|<1$ if
	\begin{eqnarray*}
		q_{1}, q_{2}\in\left[p-r(p)-\frac{M(p)}{R},p+r(p)+\frac{M(p)}{R}\right] && |q_{1}-q_{2}|<\delta
	\end{eqnarray*}

	Define $w(x):=v+\frac{M(p)}{R}\sqrt{|x|^{2}+1}+\frac{M(p)}{\delta(p)R\lambda}$, then $|w^{\prime}|\leq r(p)+\frac{M(p)}{R}$.
	
	Thus $H(p+w^{\prime},x,\omega)$$\geq$$ H(p+v^{\prime},x,\omega)-\frac{M(p)}{\delta(p)R}$ and we have the following.		
	\begin{eqnarray*}
		\lambda w+H(p+w^{\prime},x,\omega)&=&\lambda v+\frac{\lambda M(p)}{R}\sqrt{|x|^{2}+1}+\frac{M(p)}{\delta(p)R}+H(p+w^{\prime},x,\omega)\\
		&\geq& \lambda v+\frac{\lambda M(p)}{R}\sqrt{|x|^{2}+1}+\frac{M(p)}{\delta(p)R}+H(p+v^{\prime},x,\omega)-\frac{M(p)}{\delta(p)R}\\
		&>& 0
	\end{eqnarray*}
	
	Moreover,
	\begin{eqnarray*}
		|\lambda u|+|\lambda v|\leq M(p)\implies  w|_{|x|=\frac{R}{\lambda}}\geq v|_{|x|=\frac{R}{\lambda}}+\frac{M(p)}{\lambda}
		\geq u|_{|x|=\frac{R}{\lambda}}
	\end{eqnarray*}

	By comparison principle, $w(x)\geq u(x),\ x\in \left[-\frac{R}{\lambda},\frac{R}{\lambda}\right]$. So
	\begin{eqnarray*}
	u-v\leq \frac{M(p)}{R}\sqrt{|x|^{2}+1}+\frac{M(p)}{\delta(p)R\lambda}
	\end{eqnarray*}
	
	Similarly,
	\begin{eqnarray*}
	v-u\leq \frac{M(p)}{R}\sqrt{|x|^{2}+1}+\frac{M(p)}{\delta(p)R\lambda}
	\end{eqnarray*}

	Thus when $\lambda\leq 1$, let $C(p)=\frac{1}{\delta(p)}$, then for $ x\in \left[-\frac{R}{\lambda},\frac{R}{\lambda}\right]$, we have
	\begin{eqnarray*}
		&&|\lambda u-\lambda v|\leq \frac{\lambda M(p)\sqrt{|x|^{2}+1}}{R}+\frac{C(p)M(p)}{R}
		\leq \frac{M(p)\sqrt{|x|^{2}+1}}{R}+\frac{C(p)M(p)}{R}
	\end{eqnarray*}
\end{proof}

\section{Reduction by constrained Hamiltonian with index $(\widetilde{L},L)$}\label{reduction to constrained Hamiltonian}

\subsection{Approximation by cluster-point-free Hamiltonians}

Let $H(p,x,\omega)$ satisfy \textbf{(A1)-(A3)} and denote
\begin{eqnarray*}
	h_{i}^{(n)}(x,\omega):=H\left(\frac{i}{n},x,\omega\right) && \mathcal{E}_{n} = \lbrace h_{i}^{(n)}(x,\omega) \rbrace_{-n^{2}\leq i\leq n^{2}}
\end{eqnarray*}
Let $\mathcal{E}_{n}^{+}=\lbrace g_{i}^{(n)}(x,\omega)\rbrace_{-n^{2}\leq i\leq n^{2}}$ be another family of stationary functions and define
\begin{eqnarray*}
	\Delta_{\mathcal{E}_{n},\mathcal{E}_{n}^{+}}(p,x,\omega):=\begin{cases}
		g_{-n^{2}}^{(n)}-h_{-n^{2}}^{(n)}  & p\in(-\infty,-n)\\
		(np-i)\left[g_{i+1}^{(n)}-h_{i+1}^{(n)}\right]+(i+1-np)\left[g_{i}^{(n)}-h_{i}^{(n)}\right]  & p\in\left[\frac{i}{n},\frac{i+1}{n}\right]\\
		g_{n^{2}}^{(n)}-h_{n^{2}}^{(n)} & p\in(n,+\infty)\\
	\end{cases} 
\end{eqnarray*}
So $\Delta_{\mathcal{E},\mathcal{E}^{+}}(p,x,\omega)$ is a stationary function and is continuous with respect to $(p,x)$.

\begin{lemma}\label{approximation by cluster point free Hamiltonians}
	If $H(p,x,\omega)$ satisfies \textbf{(A1)-(A3)}, then there is $\lbrace H^{(n)}(p,x,\omega)\rbrace_{n=2^{j}, j\in\NN}$, such that
	
	(1)$H^{(n)}(p,x,\omega)$ satisfies \textbf{(A1)-(A3)}, $\forall n=2^{j}, j\in\NN$.
	
	(2)For $-n^{2}\leq i\leq n^{2}$, $H^{(n)}\left(\frac{i}{n},x,\omega\right)$, as functions of $x$, have no cluster point. 
	
	(3)$\lVert H(p,x,\omega)-H^{(n)}(p,x,\omega) \rVert_{L^{\infty}(\RR\times\RR\times\Omega)}\leq \frac{1}{n}$.

\end{lemma}

\begin{proof}
	For each $\epsilon>0, -n^{2}\leq i\leq n^{2}$, denote
	\begin{eqnarray*}
		& H_{\epsilon}\left(\frac{i}{n},x,\omega\right):=\frac{1}{\sqrt{2\pi \epsilon}}\int_{\RR} e^{-\frac{(x-y)^{2}}{2\epsilon}}H\left(\frac{i}{n},y,\omega\right)dy&\\
	\end{eqnarray*}
	

	By \textbf{(A1)}, either $H_{\epsilon}\left(\frac{i}{n},x,\omega\right)$ has no cluster for a.e. $\omega\in\Omega$ or $H_{\epsilon}\left(\frac{i}{n},x,\omega\right)$ has cluster for a.e. $\omega\in\Omega$.

	\vspace{3mm}
	
	If $H_{\epsilon}\left(\frac{i}{n},x,\omega\right)$ has a cluster point $x_{0}$, without loss of generality, we can assume $x_{0}=0$ and $H_{\epsilon}\left(\frac{i}{n},0,\omega\right)=0$. Then $\frac{\partial^{k}}{\partial x^{k}}H_{\epsilon}\left(\frac{i}{n},0,\omega\right)=0, k=0,1,2,\cdots$, which means
	\begin{eqnarray*}
		\int_{\RR}y^{k}e^{-\frac{y^{2}}{2\epsilon}}H\left(\frac{i}{n},y,\omega \right)dy=0
	\end{eqnarray*}
	
	By Fourier analysis, we have $H\left(\frac{i}{n},x,\omega \right)\equiv 0$. Denote
	\begin{eqnarray*}
		D_{j}:=\left\lbrace \frac{i}{2^{j}} \right\rbrace_{-4^{j}\leq i\leq 4^{j}} && D:=\bigcup_{j=1}^{\infty} D_{j}
	\end{eqnarray*}
	
	Then $D$ is dense in $\RR$, if $H_{\epsilon}\left(d,x,\omega\right)$ has a cluster point for all $d\in D$, then $H(d,x,\omega)$ is independent of $(x,\omega)$ for all $d\in D$. By continuity, $H(p,x,\omega)$ is independent of $(x,\omega)$ for all $p\in\RR$, so it is already homogenized.

	Thus, assume for some $d_{0}\in D$, $H_{\epsilon}(d_{0},x,\omega)$ has no cluster point. Since $D_{j}$ is increasing, without loss of generality, assume $d_{0}=0\in D_{j}, j\in\NN$. For $j\in\NN$ and $-4^{j}\leq i\leq 4^{j}$, define
	\begin{eqnarray*}
		g_{i}^{(2^{j})}(x,\omega):=\begin{cases}
			H\left(\frac{i}{2^{j}},x,\omega\right) & \text{ if }H\left(\frac{i}{2^{j}},x,\omega\right)\text{ has no cluster point}\\
			H\left(\frac{i}{2^{j}},x,\omega\right)+\frac{1}{2^{j}}\frac{H_{\epsilon}\left(0,x,\omega\right)}{\lVert H_{\epsilon}\left(0,x,\omega\right) \rVert_{L^{\infty}(\RR\times\Omega)}+1} & \text{ if }H\left(\frac{i}{2^{j}},x,\omega\right)\text{ has a cluster point}\\
		\end{cases}
	\end{eqnarray*}

	Denote
	\begin{eqnarray*}
		\mathcal{E}_{2^{j}}:=\lbrace h_{i}^{(2^{j})}(x,\omega) \rbrace_{-4^{j}\leq i\leq 4^{j}} && \mathcal{E}_{2^{j}}^{+}:=\lbrace g_{i}^{(2^{j})}(x,\omega) \rbrace_{-4^{j}\leq i\leq 4^{j}}
	\end{eqnarray*}

	We can finish the proof by defining
	\begin{eqnarray*}
		H^{(n)}(p,x,\omega):=H(p,x,\omega)+\Delta_{\mathcal{E}_{2^{j}},\mathcal{E}_{2^{j}}^{+}}(p,x,\omega) && n=2^{j}
	\end{eqnarray*}

\end{proof}

\subsection{Approximation by constrained Hamiltonians}
In this subsection, we find a way to approximate $H(p,x,\omega)$ by $\lbrace H_{n}(p,x,\omega)\rbrace_{n\geq 1}$ in the sense of Lemma \ref{Stability}. Here each $H_{n}(p,x,\omega)$ is constrained in the following sense.

\begin{definition}[Constrained Hamiltonian]\label{constrained Hamiltonian}
	A Hamiltonian $H(p,x,\omega)$ is called constrained if it satisfies the following (1)-(5).
	
	(1)There is $k\in\NN$ and $-\infty<a_{1}<b_{1}<a_{2}<b_{2}<\cdots<a_{k-1}<b_{k-1}<a_{k}<+\infty$.
	
	(2)For each $(x,\omega)$, $H(p,x,\omega)|_{(-\infty,a_{1})}$, $H(p,x,\omega)|_{(b_{1},a_{2})}$,$\cdots$,$H(p,x,\omega)|_{(b_{k-1},a_{k})}$ are decreasing.
	
	(3)For each $(x,\omega)$, $H(p,x,\omega)|_{(a_{k},+\infty)}$, $H(p,x,\omega)|_{(a_{k-1},b_{k-1})}$,$\cdots$,$H(p,x,\omega)|(a_{1},b_{1})$ are increasing.
	
	(4)$H(p,x,\omega)$ is Lipschitz with respect to $p$(with Lipschitz constant $\mathcal{L}$), uniformly in $(x,\omega)\in\RR$.
	
	(5)Each of $H(a_{i},x,\omega),H(b_{j},x,\omega),1\leq i\leq k, 1\leq j\leq k-1$ has no cluster point.
\end{definition}

\begin{lemma}\label{approximate general coercive constrained Hamiltonian}
	If $H(p,x,\omega)$ satisfies \textbf{(A1)-(A3)}, then for $n=2^{j}, j\in\NN$, there is $H_{n}(p,x,\omega)$, such that

	(a)$\lbrace H(p,x,\omega) \rbrace_{n\geq 1}$ is uniformly coercive.
	
	(b)Each $H_{n}(p,x,\omega)$ satisfies \textbf{(A1)-(A3)}.

	(c)Each $H_{n}(p,x,\omega)$ is constrained.

	(d)Fix any $\delta>0$. Then for any compact set $K\subset\RR$, there is an $N\in\NN$,
	\begin{eqnarray*}
	\lVert H_{n}(p,x,\omega) - H(p,x,\omega) \rVert_{L^{\infty}(K\times\RR\times\Omega)}<\delta && \text{ if }n>N
	\end{eqnarray*}

\end{lemma}
\begin{proof}
	According to Lemma \ref{approximation by cluster point free Hamiltonians}, without loss of generality, we can assume each of $H\left(\frac{i}{n},x,\omega\right)$, $-n^{2}\leq i\leq n^{2}$ has no cluster point.  We construct $H_{n}(p,x,\omega)$ by the following procedure.

	STEP 1: For each $p\in(-\infty,n)\cup(n,\infty)$, define
	\begin{eqnarray*}
		H_{n}(p,x,\omega)=\begin{cases}
			|p+n|+H(-n,x,\omega) & p\in(-\infty,-n)\\
			|p-n|+H(n,x,\omega) & p\in(n,+\infty)
		\end{cases}
	\end{eqnarray*}

	STEP 2: For $k=0,1,2,\cdots,2n^{2}$, define
	\begin{eqnarray*}
		H_{n}\left(-n+\frac{k}{n},x,\omega\right)=H\left(-n+\frac{k}{n},x,\omega\right)
	\end{eqnarray*}
	
	STEP 3: For $i=0,1,2,\cdots,2n^{2}-1$, define
	\begin{eqnarray*}
		H_{n}\left(-n+\frac{i}{n}+\frac{1}{2n},x,\omega\right)=\max\left\lbrace H\left(-n+\frac{i}{n},x,\omega\right), H\left(-n+\frac{i+1}{n},x,\omega\right) \right\rbrace+\frac{1}{n}
	\end{eqnarray*}

	STEP 4: For $i=0,1,2,\cdots,2n^{2}-1$, 
	
	(1)If $p\in\left(-n+\frac{i}{n}, -n+\frac{i}{n}+\frac{1}{2n}\right)$, then there is some $\theta\in(0,1)$, such that 
	\begin{eqnarray*}
		p=\theta\times \left(-n+\frac{i}{n}\right)+(1-\theta)\times\left(-n+\frac{i}{n}+\frac{1}{2n}\right)
	\end{eqnarray*}
	
	Then we define 
	\begin{eqnarray*}
		H_{n}(p,x,\omega) = \theta H\left(-n+\frac{i}{n},x,\omega\right)+(1-\theta)H\left(-n+\frac{i}{n}+\frac{1}{2n},x,\omega\right)
	\end{eqnarray*}

	(2)If $p\in\left(-n+\frac{i}{n}+\frac{1}{2n}, -n+\frac{i+1}{n}\right)$, then there is some $\theta\in(0,1)$, such that 
	\begin{eqnarray*}
		p=\theta\times \left(-n+\frac{i}{n}+\frac{1}{2n}\right)+(1-\theta)\times\left(-n+\frac{i+1}{n}\right)
	\end{eqnarray*}
	
	Then we define 
	\begin{eqnarray*}
		H_{n}(p,x,\omega) = \theta H\left(-n+\frac{i}{n}+\frac{1}{2n},x,\omega\right)+(1-\theta)H\left(-n+\frac{i+1}{n},x,\omega\right)
	\end{eqnarray*}

	(a)Since $H(p,x,\omega)$ satisfies \textbf{(A2)}, $\lbrace H_{n}(p,x,\omega) \rbrace_{n\geq 1}$ is uniformly coercive.

	(b)By \textbf{(A1)}, $H\left(-n+\frac{k}{n},x,\omega\right)$ is stationary, for $k = 0,1,2,\cdots,2n^{2}$. So, $H_{n}(p,x,\omega)$, as a linear combination of these functions, is stationary and satisfies \textbf{(A1)-(A3)}.
	
	(c)By the above construction, such $H_{n}(p,x,\omega)$ is constrained with $2n^{2}+1$ wells. And $H_{n}(p,x,\omega)$ has Lipschitz constant $\mathcal{L}=1+n\rho_{[-n^{2},n^{2}]}(\frac{1}{n})$ in $p$ variable, uniformly in $(x,\omega)\in\RR$.
	
	(d)By \textbf{(A3)}, there is $N\in\NN$, such that $N>\frac{3}{\delta}$, $K\subset[-N,N]$ and
	\begin{eqnarray*}
		p,q\in K,|p-q|<\frac{1}{N}\implies |H(p,x,\omega)-H(q,x,\omega)|<\frac{\delta}{3}
	\end{eqnarray*}
	
	To prove (d), it suffices to show that: fix any $k\in\lbrace 0,1,2,\cdots,2n^{2}-1 \rbrace$
	\begin{eqnarray*}
		\lVert H_{n}(p,x,\omega)-H(p,x,\omega) \rVert_{L^{\infty}((K\cap(-n+\frac{k}{n},-n+\frac{k+1}{n}))\times\RR)}<\delta
	\end{eqnarray*}

	Denote
	\begin{eqnarray*}
		p_{1}=-n+\frac{k}{n},\ p_{2}=-n+\frac{k}{n}+\frac{1}{2n},\ p_{3}=-n+\frac{k+1}{n}
	\end{eqnarray*}
	
	Without loss of generality, assume that
	\begin{eqnarray*}
		H\left(p_{1},x,\omega\right)\leq H\left(p_{3},x,\omega\right)
	\end{eqnarray*}
	
	\textbf{Case 1:} $p\in K\cap(p_{1},p_{2})$. Then there is some $\theta\in(0,1)$ such that $p=\theta p_{1}+(1-\theta)p_{2}$,
	\begin{eqnarray*}
		|H_{n}(p,x,\omega)-H(p,x,\omega)|&=&|H_{n}(\theta p_{1}+(1-\theta)p_{2},x,\omega)-H(\theta p_{1}+(1-\theta)p_{2},x,\omega)|\\
		&=& \Big|\theta H(p_{1},x,\omega)+(1-\theta)\left[H(p_{3},x,\omega)+\frac{1}{n}\right]\\
		&& \ \ \ -H(\theta p_{1}+(1-\theta)p_{2},x,\omega)\Big|\\
		&\leq& \theta |H(p_{1},x,\omega)- H(\theta p_{1}+(1-\theta)p_{2},x,\omega)|\\
		&& +(1-\theta)|H(p_{3},x,\omega)-H(\theta p_{1}+(1-\theta)p_{2},x,\omega)|+\frac{1-\theta}{n}\\
		&<&\delta
	\end{eqnarray*}

	\textbf{Case 2:} $p\in K\cap(p_{2},p_{3})$. Then there is some $\theta\in(0,1)$ such that $p=\theta p_{2}+(1-\theta)p_{3}$,
	\begin{eqnarray*}
		|H_{n}(p,x,\omega)-H(p,x,\omega)|&=&|H_{n}(\theta p_{2}+(1-\theta)p_{3},x,\omega)-H(\theta p_{2}+(1-\theta)p_{3},x,\omega)|\\
		&=& \Big|\theta \left[H(p_{3},x,\omega)+\frac{1}{n}\right]+(1-\theta)H(p_{3},x,\omega)\\
		&& \ \ \ -H(\theta p_{1}+(1-\theta)p_{2},x,\omega)\Big|\\
		&\leq& |H(p_{3},x,\omega)-H(\theta p_{1}+(1-\theta)p_{2},x,\omega)|+\frac{\theta}{n}\\
		&<&\delta
	\end{eqnarray*}

	The above is true for all $k=0,1,2,\cdots,2n^{2}-1$, thus
	\begin{equation*}
	\lVert H_{n}(p,x,\omega) - H(p,x,\omega) \rVert_{L^{\infty}(K\times\RR\times\RR)}<\delta
	\end{equation*} 
	
\end{proof}

\begin{remark}
	By Lemma \ref{Stability} and Lemma \ref{approximate general coercive constrained Hamiltonian}, to prove Theorem \ref{main theorem}, it suffices to consider such Hamiltonian $H(p,x,\omega)$ that is constrained(\ref{constrained Hamiltonian}) and satisfies \textbf{(A1)-(A3)}. So in the following sections, we only consider constrained Hamiltonians.
\end{remark}

\subsection{Constrained Hamiltonian with index $(\widetilde{L},L)$}
\begin{definition}
	$H(p,x,\omega)$ is called constrained Hamiltonian with index $(\widetilde{L},L)$ if
	
	(1)$H(p,x,\omega)$ is constrained (\ref{constrained Hamiltonian}).
	
	(2)$(a_{1},b_{1},a_{2},b_{2},\cdots,a_{k-1},b_{k-1},a_{k})=(\widetilde{p}_{1},\widetilde{q}_{1},\widetilde{p}_{2},\widetilde{q}_{2},\cdots,\widetilde{p}_{\widetilde{L}},\widetilde{q}_{\widetilde{L}},0,q_{L},p_{L},q_{L-1},p_{L-1},\cdots,q_{1},p_{1})$.
	
	(3)$\esssup\limits\limits_{(x,\omega)}H(\widetilde{p}_{i},x,\omega)>0, 1\leq i\leq \widetilde{L}$; $\esssup\limits\limits_{(x,\omega)}H(0,x,\omega)=0$; $\esssup\limits\limits_{(x,\omega)}H(p_{j},x,\omega)>0,1\leq j\leq L$.
	
	(4)Each of $H(a_{i},x,\omega), H(b_{i},x,\omega), 1\leq i\leq k$ has no cluster point.

\end{definition}

\begin{remark}
	Apply perturbation and shift coordinates if necessary, it suffices to consider homogenization of any constrained Hamiltonian with index $(\widetilde{L},L)$. The following example is a constrained Hamiltonian with index $(1,2)$.
\end{remark}

\begin{figure}[h]
\centering
\includegraphics[width=0.4\linewidth]{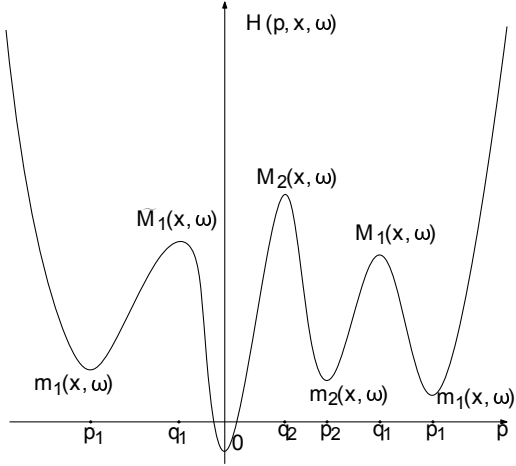}
\caption{Constrained Hamiltonian with index $(1,2)$}
\label{fig:TheHamiltonian4Wells_withcoord}
\end{figure}

\begin{notation}
	
	Let $H(p,x,\omega)$ be a constrained Hamiltonian with index $(\widetilde{L},L)$.
	
	(1)For each $(x,\omega)$, denote monotone branches of $H(p,x,\omega)$ by
	\begin{eqnarray*}
		& H|_{[p_{1},\infty)}:=\phi_{1,(x,\omega)}(p),\
		H|_{[q_{1},p_{1}]}:=\phi_{2,(x,\omega)}(p),\ \cdots,\ H|_{[0,q_{L}]}:=\phi_{2L+1,(x,\omega)}(p)&\\
		& H|_{[-\infty,\widetilde{p}_{1})}:=\widetilde{\phi}_{1,(x,\omega)}(p),\ H|_{[\widetilde{p}_{1},\widetilde{q}_{1}]}:=\widetilde{\phi}_{2,(x,\omega)}(p),\ \cdots,\ H|_{[\widetilde{q}_{\widetilde{L}},0]}:=\widetilde{\phi}_{2\widetilde{L}+1,(x,\omega)}(p)&
	\end{eqnarray*}

(2)Denote inverse function of each branch by 
	\begin{eqnarray*}
		\left(\phi_{i,(x,\omega)}(\cdot)\right)^{-1}:=\psi_{i,(x,\omega)}(\cdot) && \left(\widetilde{\phi}_{i,(x,\omega)}(\cdot)\right)^{-1}:=\widetilde{\psi}_{i,(x,\omega)}(\cdot)
	\end{eqnarray*}

(3)Denote local extreme values by
	\begin{eqnarray*}
		m_{i}(x,\omega):= H(p_{i},x,\omega)&&
		\widetilde{m}_{i}(x,\omega):= H(\widetilde{p}_{i},x,\omega)\\
		M_{i}(x,\omega):= H(q_{i},x,\omega)&&
		\widetilde{M}_{i}(x,\omega):= H(\widetilde{q}_{i},x,\omega)
	\end{eqnarray*}
	
(4)Define two functions	
	\begin{eqnarray}
	\label{m}m(x,\omega)&:=&\min\left\lbrace \min_{1\leq i\leq L}m_{i}(x,\omega), \min_{1\leq j\leq\widetilde{L}}\widetilde{m}_{j}(x,\omega) \right\rbrace\\
	\label{M}M(x,\omega)&:=&\max\left\lbrace \max_{1\leq i\leq L}M_{i}(x,\omega), \max_{1\leq j\leq\widetilde{L}}\widetilde{M}_{j}(x,\omega) \right\rbrace
	\end{eqnarray}

\end{notation}

\section{Auxiliary Lemmas for Gluing Lemmas}\label{Auxiliary Lemmas for Gluing Lemmas}
\subsection{Estimation of Gradient}

\begin{lemma}\label{Two points control}
Let Hamiltonian $H(p,x,\omega)$ satisfy $\textbf{(A1)-(A3)}$ and be regularly homogenizable at $p_{0}$, for each $\lambda>0$, let $v_{\lambda}(x,p_{0},\omega)$ be the viscosity solution of the equation:
\begin{eqnarray*}
\lambda v_{\lambda}+H(p_{0}+v_{\lambda}^{\prime},x,\omega)=0
\end{eqnarray*}
fix $P\in\RR$, denote $\underline{P}:=\essinf\limits\limits_{(x,\omega)} H(P,x,\omega)$ and $\overline{P}:=\esssup\limits\limits_{(x,\omega)} H(P,x,\omega)$, then, there is an $\widetilde{\Omega}\subset\Omega$ with $\mathbf{P}[\widetilde{\Omega}]=1$, such that, for each $\omega\in\widetilde{\Omega}$, the following hold.

$(1)$If $\overline{H}(p_{0})<\underline{P}$, $p_{0}<P$, then for any $R>0$, there is $\lambda_{0}=\lambda_{0}(R,p_{0},\omega)>0$,
\begin{eqnarray*}
0<\lambda<\lambda_{0}\implies p_{0}+v_{\lambda}^{\prime}(x,p_{0},\omega)\leq P & x\in\left[-\frac{R}{\lambda},\frac{R}{\lambda}\right]
\end{eqnarray*}

$(2)$If $\overline{H}(p_{0})<\underline{P}$, $p_{0}>P$, then for any $R>0$, there is $\lambda_{0}=\lambda_{0}(R,p_{0},\omega)>0$,
\begin{eqnarray*}
0<\lambda<\lambda_{0}\implies p_{0}+v_{\lambda}^{\prime}(x,p_{0},\omega)\geq P& x\in\left[-\frac{R}{\lambda},\frac{R}{\lambda}\right]
\end{eqnarray*}

$(3)$If $\overline{H}(p_{0})>\overline{P}$, $p_{0}<P$, then for any $R>0$, there is $\lambda_{0}=\lambda_{0}(R,p_{0},\omega)>0$,
\begin{eqnarray*}
0<\lambda<\lambda_{0}\implies p_{0}+v_{\lambda}^{\prime}(x,p_{0},\omega)\leq P & x\in\left[-\frac{R}{\lambda},\frac{R}{\lambda}\right]
\end{eqnarray*}

$(4)$If $\overline{H}(p_{0})>\overline{P}$, $p_{0}>P$, then for any $R>0$, there is $\lambda_{0}=\lambda_{0}(R,p_{0},\omega)>0$, 
\begin{eqnarray*}
0<\lambda<\lambda_{0}\implies p_{0}+v_{\lambda}^{\prime}(x,p_{0},\omega)\geq P & x\in\left[-\frac{R}{\lambda},\frac{R}{\lambda}\right]
\end{eqnarray*}
\end{lemma}

\begin{proof}[Proof of periodic case]
	(1)For $p_{0}$, we have the cell problem
	\begin{equation*}
	H(p_{0}+v^{\prime},x)=\overline{H}(p_{0})
	\end{equation*}
	
	Suppose (1) is not true, then there is $x_{1}\in[0,1]$, such that $p_{0}+v^{\prime}(x_{1})>P$. On the other hand,	
	\begin{equation*}
	\int_{1}^{2}p_{0}+v^{\prime}(x)-Pdx=p_{0}-P<0
	\end{equation*}
	
	So there is some $y_{1}\in[1,2]$, such that $p_{0}+v^{\prime}(y_{1})-P<0$. Then $\Psi(x):=p_{0}+v(x)-Px$ attains local maximum at some $z_{1}\in[x_{1},y_{1}]$, so
	\begin{eqnarray*}
	\underline{P}\leq H(P,z_{1})\leq \overline{H}(p_{0})<\underline{P}
	\end{eqnarray*}
	
	This is a contradiction, so we proved (1). The proofs of (2)(3)(4) are similar.
\end{proof}

\begin{proof}[Proof of random case]
(1)If it is not true, then there is $\Omega_{1}\subset\Omega$, $\mathbf{P}[\Omega_{1}]>0$, for any $\omega\in\Omega_{1}$, there are $R_{1}=R_{1}(p_{0},\omega)>0$ and $\lambda_{n}\rightarrow 0$ such that
\begin{eqnarray*}
p_{0}+v_{\lambda_{n}}^{\prime}(x_{\lambda_{n}},p_{0},\omega)>P &\text{ for some } x_{\lambda_{n}}\in\left[-\frac{R_{1}}{\lambda_{n}},\frac{R_{1}}{\lambda_{n}}\right]
\end{eqnarray*}

Denote $\delta:=P-p_{0}>0$. For any $R>0$, we have
\begin{eqnarray*}
\left|\frac{\lambda}{R}\int_{\frac{R_{1}}{\lambda}}^{\frac{R+R_{1}}{\lambda}}v_{\lambda}^{\prime}(s,\omega)ds\right|\leq \frac{2\big(H_{\sup}(p_{0})-H_{\inf}(p_{0})\big)}{R}
\end{eqnarray*}

Fix any $R_{2}=R_{2}(p_{0})>\frac{4(H_{\sup}(p_{0})-H_{\inf}(p_{0}))}{\delta}$, thus for any $R\geq R_{2}$, we have
\begin{eqnarray*}
\left|\frac{\lambda}{R}\int_{\frac{R_{1}}{\lambda}}^{\frac{R+R_{1}}{\lambda}}v_{\lambda}^{\prime}(s,p_{0},\omega)ds\right|<\frac{\delta}{2} 
\text{ for any } \lambda>0
\end{eqnarray*}

So
\begin{eqnarray*}
\frac{\lambda_{n}}{R_{2}}\int_{\frac{R_{1}}{\lambda_{n}}}^{\frac{R_{2}+R_{1}}{\lambda_{n}}}p_{0}+v_{\lambda_{n}}^{\prime}(s,p_{0},\omega)-Pds\leq p_{0}-P+\frac{\delta}{2}<0
\end{eqnarray*} 

This implies
\begin{eqnarray*}
p_{0}+v_{\lambda_{n}}^{\prime}(y_{\lambda_{n}},\omega)-P<0 & \text{ for some } y_{\lambda_{n}}\in \left(\frac{R_{1}}{\lambda_{n}},\frac{R_{2}+R_{1}}{\lambda_{n}}\right)
\end{eqnarray*}

Denote $\Psi(x,\omega)=p_{0}x+v_{\lambda_{n}}(x,\omega)-Px$, then
\begin{eqnarray*}
\Psi(x,\omega) \text{ is increasing (decreasing) in a neighborhood of } x_{\lambda_{n}}(y_{\lambda_{n}})
\end{eqnarray*}

Since $x_{\lambda_{n}}<y_{\lambda_{n}}$, $\Psi(x,\omega)$ attains local maximum at some  $z_{\lambda_{n}}\in\left(x_{\lambda_{n}},y_{\lambda_{n}}\right)$. So
\begin{eqnarray}\label{local max ineq}
\lambda_{n}v_{\lambda_{n}}(z_{\lambda_{n}},\omega)+H(P,z_{\lambda_{n}},\omega)\leq 0 
\end{eqnarray}

Since $H(p,x,\omega)$ is regularly homogenizable at $p_{0}$, there is $\Omega_{2}\subset\Omega$, s.t. $\mathbf{P}[\Omega_{2}]=1$, 
\begin{eqnarray*}
\limsup_{\lambda\rightarrow 0}\sup_{|x|\leq\frac{R_{1}+R_{2}}{\lambda}}|\lambda v_{\lambda}(x,\omega)+\overline{H}(p_{0})|=0 \text{ for each }\omega\in \Omega_{2}
\end{eqnarray*}

Denote $\tau:=\underline{P}-\overline{H}(p_{0})>0$, $\hat{\Omega}:=\Omega_{1}\cap\Omega_{2}$. So there is some $N_{1}(\omega)$, 
\begin{eqnarray}\label{cover to bar H}
\sup_{|x|\leq\frac{R_{1}+R_{2}}{\lambda_{n}}}|\lambda_{n}v_{\lambda_{n}}+\overline{H}(p_{0})|<\frac{\tau}{2} \text{ for any }n\geq N_{1}
\end{eqnarray}
\begin{eqnarray*}
\mathbf{P}[\Omega_{1}]>0,
\mathbf{P}[\Omega_{2}]=1 \implies \mathbf{P}[\hat{\Omega}]>0\implies \hat{\Omega}\ne\emptyset
\end{eqnarray*}

Choose any $\omega\in\hat{\Omega}$ and $n\geq N_{1}(\omega)$, by (\ref{local max ineq}) and (\ref{cover to bar H}),
\begin{eqnarray*}
& \underline{P}\leq H(P,z_{\lambda_{n}},\omega)\leq -\lambda_{n}v_{\lambda_{n}}(z_{\lambda_{n}},\omega)
\leq \overline{H}(p_{0})+\frac{\tau}{2}
= \underline{P}-\tau +\frac{\tau}{2}
=\underline{P}-\frac{\tau}{2}
\end{eqnarray*}

This is a contradiction. Thus (1) is proved. The proofs of (2)(3)(4) are similar.

\end{proof}

\begin{lemma}\label{Three points control}
Let Hamiltonian $H(p,x,\omega)$ satisfy \textbf{(A1)-(A3)} and be regularly homogenizable at $p_{0}\in\RR$ to $\overline{H}(p_{0})$, for each $\lambda$, let $v_{\lambda}(x)$ be the viscosity solution of the following equation:
\begin{eqnarray*}
\lambda v_{\lambda}(x)+H(p_{0}+v_{\lambda}^{\prime}(x),x,\omega)=0
\end{eqnarray*}
For $P,Q\in\RR$, denote 
\begin{eqnarray*}
\underline{P}:=\essinf\limits\limits_{(x,\omega)} H(P,x,\omega) &&
\overline{P}:=\esssup\limits\limits_{(x,\omega)} H(P,x,\omega)\\
\underline{Q}:=\essinf\limits\limits_{(x,\omega)} H(Q,x,\omega) &&
\overline{Q}:=\esssup\limits\limits_{(x,\omega)} H(Q,x,\omega)
\end{eqnarray*}
Then, there is an $\widetilde{\Omega}\subset\Omega$ with $\mathbf{P}[\tilde{\Omega}]=1$, such that for each $\omega\in\tilde{\Omega}$, the following hold.

$(1)$If $p_{0}<P$, $P<Q$ and $\overline{P}<\underline{Q}$, then for each $R>0$, there is $\lambda_{0}=\lambda_{0}(R,p_{0},\omega)$, 
\begin{eqnarray*}
0<\lambda<\lambda_{0}\implies p_{0}+v_{\lambda}^{\prime}(x)\leq Q & x\in\left[-\frac{R}{\lambda},\frac{R}{\lambda}\right]
\end{eqnarray*}

$(2)$If $p_{0}<P$, $P<Q$ and $\underline{P}>\overline{Q}$, then for each $R>0$, there is $\lambda_{0}=\lambda_{0}(R,p_{0},\omega)$,
\begin{eqnarray*}
0<\lambda<\lambda_{0}\implies p_{0}+v_{\lambda}^{\prime}(x)\leq Q & x\in\left[-\frac{R}{\lambda},\frac{R}{\lambda}\right]
\end{eqnarray*}

$(3)$If $p_{0}>P$, $P>Q$ and $\overline{P}<\underline{Q}$, then for each $R>0$, there is $\lambda_{0}=\lambda_{0}(R,p_{0},\omega)$,
\begin{eqnarray*}
0<\lambda<\lambda_{0}\implies p_{0}+v_{\lambda}^{\prime}(x)\geq Q & x\in\left[-\frac{R}{\lambda},\frac{R}{\lambda}\right]
\end{eqnarray*}

$(4)$If $p_{0}>P$, $P>Q$ and $\underline{P}>\overline{Q}$, then for each $R>0$, there is $\lambda_{0}=\lambda_{0}(R,p_{0},\omega)$,
\begin{eqnarray*}
0<\lambda<\lambda_{0}\implies p_{0}+v_{\lambda}^{\prime}(x)\geq Q & x\in\left[-\frac{R}{\lambda},\frac{R}{\lambda}\right]
\end{eqnarray*}

\end{lemma}

\begin{proof}
(1)Case 1: $\overline{H}(p_{0})<\underline{P}$, apply (1) of Lemma \ref{Two points control} to $(p_{0},P)$.

Case 2: $\overline{H}(p_{0})>\overline{P}$, apply (3) of Lemma \ref{Two points control} to $(p_{0},P)$.

Case 3: $\overline{H}(p_{0})\in\left[\underline{P},\overline{P}\right]$, apply (1) of Lemma \ref{Two points control} to $(p_{0},Q)$.

The proofs of (2)(3)(4) are similar.

\end{proof}

\subsection{Squeeze Lemma}

\begin{figure}[h]
\centering
\includegraphics[width=0.4\linewidth]{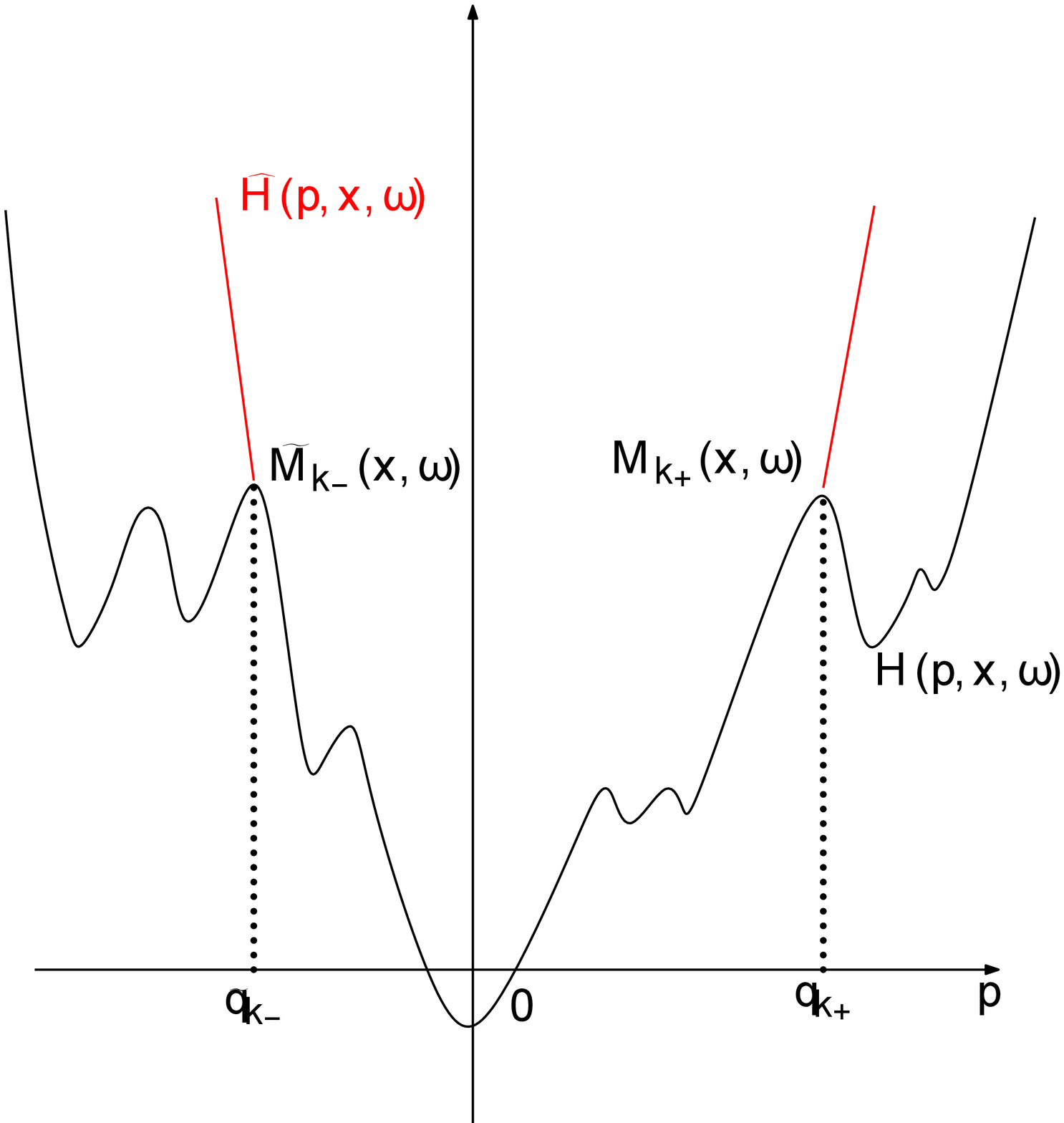}
\caption{Squeeze}
\label{fig:squeeze}
\end{figure}

\begin{lemma}\label{squeeze lemma}
Let $H(p,x,\omega)$ satisfy \textbf{(A1)-(A3)} and be constrained with index $(\widetilde{L},L)$,
if $H(p,x,\omega)$ has effective Hamiltonian $\overline{H}(p)$ with $\overline{H}(q)=0$, then the following are true.

$(1)$If $q>0$ and $\overline{H}|_{(q,+\infty)}>0$, then $\overline{H}(p)\equiv 0 \text{ for all }p\in[0,q]$.

$(2)$If $q<0$ and $\overline{H}|_{(-\infty,q)}>0$, then $\overline{H}(p)\equiv 0 \text{ for all }p\in[q,0]$.
\end{lemma}

\begin{proof}
(1)Recall the notation (\ref{m}) and $H(p,x,\omega)$ is constrained with index $(\widetilde{L},L)$, we have
\begin{eqnarray}\label{ass for squeeze}
\mathbf{E}[m(x,\omega)>0]>0
\end{eqnarray}

Denote:
\begin{eqnarray*}
\underline{M_{i}}:=\essinf\limits\limits_{(x,\omega)\in\RR\times\Omega} M_{i}(x,\omega) &&
\underline{M^{+}}:=\max\limits_{1\leq i\leq L}\underline{M_{i}}\\
\underline{\widetilde{M_{i}}}:=\essinf\limits\limits_{(x,\omega)\in\RR\times\Omega} \widetilde{M}_{i}(x,\omega) &&
\underline{M^{-}}:=\max\limits_{1\leq i\leq \widetilde{L}}\underline{\widetilde{M_{i}}}
\end{eqnarray*}

\textbf{Case 1:} $\min\left\lbrace \underline{M^{+}},\underline{M^{-}} \right\rbrace>0$. Denote
\begin{eqnarray*}
k_{+}=\max\lbrace 1\leq i\leq L|\underline{M_{i}}>0  \rbrace &&
k_{-}=\max\lbrace 1\leq i\leq \widetilde{L}|\underline{\widetilde{M_{i}}}>0  \rbrace
\end{eqnarray*}
\begin{eqnarray*}
\widehat{H}(p,x,\omega):=\begin{cases}
\mathcal{L}|p-\widetilde{q}_{k_{-}}|+H(\widetilde{q}_{k_{-}},x,\omega) & p\in (-\infty, \widetilde{q}_{k_{-}})\\
H(p,x,\omega) & p\in [\widetilde{q}_{k_{-}},q_{k_{+}}]\\
\mathcal{L}|p-q_{k_{+}}|+H(q_{k_{+}},x,\omega) & p\in(q_{k_{+}},+\infty)
\end{cases}
\end{eqnarray*}

By section \ref{Proof of Large Oscillation}, $\widehat{H}(p,x,\omega)$ has a level-set convex effective Hamiltonian $\overline{\widehat{H}}(p)\geq 0$ with $\overline{\widehat{H}}(0)=0$. For any $\lambda>0$, let $v_{\lambda}(x,q,\omega)$  and $\widehat{v}_{\lambda}(x,q,\omega)$ be solutions of the following equations:
\begin{eqnarray*}
\lambda v_{\lambda}+H(q+v_{\lambda}^{\prime},x,\omega)=0 &&  \lambda \widehat{v}_{\lambda}+\widehat{H}(q+\widehat{v}_{\lambda}^{\prime},x,\omega)=0
\end{eqnarray*}

\textbf{Claim:} $q_{k_{-}}<q<q_{k_{+}}.$

Proof of the Claim: Suppose it is not true.

(I)If $q=q_{k_{+}}$, then $0=\overline{H}(q)=\overline{H}(q_{k_{+}})\geq \underline{M_{k_{+}}}>0$, this is a contradiction.

(II)If $q>q_{k_{+}}$. The arguments are divided into the following (II-1), (II-2) and (II-3).

(II-1)By Lemma \ref{Two points control}, there is $\Omega_{1}\subset\Omega$, $\mathbf{P}[\Omega_{1}]=1$ such that if $\omega\in\Omega_{1}$, then for any $R>0$, there is $\lambda_{1}=\lambda_{1}(R,q,\omega)>0$,
\begin{eqnarray*}
0<\lambda<\lambda_{1}\implies q+v_{\lambda}^{\prime}(x,q,\omega)\geq q_{k_{+}} & x\in\left[-\frac{R}{\lambda},\frac{R}{\lambda}\right]
\end{eqnarray*}

(II-2)By (\ref{ass for squeeze}), there are $\delta>0$ and $\tau>0$ such that $\mathbf{E}\left[m(x,\omega)>\delta\right]=\tau$. By ergodic theorem, there is $\Omega_{2}\subset\Omega$, $\mathbf{P}[\Omega_{2}]=1$. For each $\omega\in\Omega_{2}$ and $R>0$,
\begin{eqnarray*}
\lim_{\lambda\rightarrow 0}\frac{2\lambda}{R}\int_{-\frac{R}{\lambda}}^{\frac{R}{\lambda}}\CCChi_{\lbrace m(\cdot,\omega)>\delta\rbrace}(x)dx=\mathbf{E}\left[m(x,\omega)>\delta\right]=\tau
\end{eqnarray*}

So there is some $\lambda_{2}(R,q,\omega)>0$, such that
\begin{eqnarray*}
0<\lambda<\lambda_{2}(R,q,\omega) \implies\frac{2\lambda}{R}\int_{-\frac{R}{\lambda}}^{\frac{R}{\lambda}}\CCChi_{\lbrace m(\cdot,\omega)>\delta\rbrace}(x)dx>\frac{\tau}{2}
\end{eqnarray*}

(II-3)Since $\overline{H}(q)=0$, there is $\Omega_{3}\subset\Omega$, $\mathbf{P}[\Omega_{3}]=1$. For each $\omega\in\Omega_{3}$ and $R>0$, there is $\lambda_{3}=\lambda_{3}(R,q,\omega)>0$,
\begin{eqnarray*}
0<\lambda<\lambda_{3} \implies\left|\lambda v_{\lambda}(x,q,\omega)
\right|<\delta  & x\in\left[-\frac{R}{\lambda},\frac{R}{\lambda}\right]
\end{eqnarray*}

Let $\widetilde{\lambda}(R,q,\omega):=\min \lbrace \lambda_{1}(R,q,\omega), \lambda_{2}(R,q,\omega),\lambda_{3}(R,q,\omega) \rbrace>0$, $\widetilde{\Omega}:=\Omega_{1}\bigcap \Omega_{2}\bigcap\Omega_{3}$, then $\mathbf{P}[\widetilde{\Omega}]=1$, for each $\omega\in\widetilde{\Omega}$, when $\lambda<\widetilde{\lambda}(R,q,\omega)$, there is $x_{\lambda}\in\left[-\frac{R}{\lambda},\frac{R}{\lambda}\right]$, $m(x_{\lambda},\omega)>\delta$,
\begin{eqnarray*}
\delta<m(x_{\lambda},\omega)\leq H(q+v_{\lambda}^{\prime}(x_{\lambda},q,\omega),x_{\lambda},\omega)=-\lambda v_{\lambda}(x_{\lambda},q,\omega)<\delta
\end{eqnarray*}

This is a contradiction. (The second inequality is because  we have $q+v_{\lambda}^{\prime}(x_{\lambda},q,\omega)\geq q_{k_{+}}$). So, $q<q_{k_{+}}$. Similarly, we can prove $q_{k_{-}}<q$. This ends the proof of the \textbf{Claim}.

\vspace{0.5cm}

By Lemma \ref{Two points control}, there is $\widehat{\Omega}$, $\mathbf{P}[\widehat{\Omega}]=1$. For $\omega\in\widehat{\Omega}$ and any $R>0$, there is $\widehat{\lambda}(R,q,\omega)>0$,
\begin{eqnarray*}
0<\lambda<\widehat{\lambda}\implies q_{k_{+}}\leq q+v_{\lambda}^{\prime}(x,q,\omega)\leq q_{k_{+}} & x\in\left[-\frac{R}{\lambda},\frac{R}{\lambda}\right]
\end{eqnarray*}

So
\begin{eqnarray*}
\lambda v_{\lambda}(x,q,\omega)+\widehat{H}(q+v_{\lambda}^{\prime}(x,q,\omega),x,\omega)=0 & x\in\left[-\frac{R}{\lambda},\frac{R}{\lambda}\right]
\end{eqnarray*}

By Lemma \ref{CP}, there is some constant $C>0$, such that
\begin{eqnarray*}
|\lambda v_{\lambda}(0,q,\omega)-\lambda \widehat{v}_{\lambda}(0,q,\omega)|\leq\frac{C}{R}
\end{eqnarray*}

$R>0$ can be chosen arbitrarily large, then
\begin{equation*}
\overline{\widehat{H}}(q)=\lim_{\lambda\rightarrow 0}-\lambda \widehat{v}_{\lambda}(0,q,\omega)=\lim_{\lambda\rightarrow 0}-\lambda v_{\lambda}(0,q,\omega)=\overline{H}(q)=0
\end{equation*}

By level-set convexity of $\overline{\widehat{H}}(p)$ and $\overline{\widehat{H}}(0)=0$,  $\overline{\widehat{H}}|_{[0,q]}\equiv 0$. Since $\widehat{H}(p,x,\omega)\geq H(p,x,\omega)$, $\overline{\widehat{H}}(p)\geq \overline{H}(p)$. On the other hand, $\overline{H}(p)\geq 0$. So $\overline{H}|_{[0,q]}\equiv 0$.

\textbf{Case 2:} $\min\left\lbrace \underline{M^{+}},\underline{M^{-}} \right\rbrace\leq 0<\max\left\lbrace \underline{M^{+}},\underline{M^{-}} \right\rbrace$.

Construct $\widehat{H}(p,x,\omega)$ by modifying one side and similar arguments thereafter.

\textbf{Case 3:} $\max\left\lbrace \underline{M^{+}},\underline{M^{-}} \right\rbrace\leq 0$.

By section \ref{Proof of Large Oscillation}, $\overline{H}(p)$ is level-set convex. Define $\widehat{H}(p,x,\omega):=H(p,x,\omega)$ and apply the result of \textbf{Case 1}. This ends the proof of (1).

\vspace{5mm}
The proof for (2) is similar.

\end{proof}

\section{Reduction by constrained Hamiltonian with index $(\widetilde{L},0)$ and $(0,L)$}\label{reduction to one side constrained Hamiltonians}

Let $H(p,x,\omega)$ be a constrained Hamiltonian that satisfies \textbf{(A1)-(A3)}. Define

\begin{eqnarray*}
H^{+}(p,x,\omega):=\begin{cases}
H(p,x,\omega) & p\geq 0\\
\mathcal{L}|p|+H(0,x,\omega) & p< 0
\end{cases}&& H^{-}(p,x,\omega):=\begin{cases}
\mathcal{L}|p|+H(0,x,\omega) & p\geq 0\\
H(p,x,\omega) & p< 0
\end{cases}
\end{eqnarray*}

\begin{lemma}\label{gluing at minimum point}
	If both $H^{+}(p,x,\omega)$ and $H^{-}(p,x,\omega)$ are regularly homogenizable for all $p\in\RR$, then $H(p,x,\omega)$ is also regularly homogenizable for all $p\in\RR$ and
	\begin{eqnarray*}
	\overline{H}(p)=\begin{cases}
	\overline{H^{+}}(p) & p\geq 0\\
	\overline{H^{-}}(p) & p<0
	\end{cases}
	\end{eqnarray*}
\end{lemma}

\begin{proof}
	Fix $p\geq 0$, $\omega\in\Omega$ and $\lambda>0$, let $v_{\lambda}(x,p,\omega)$ and $v_{+,\lambda}(x,p,\omega)$ be solutions of the equations
	\begin{eqnarray*}
		\lambda v_{\lambda}+H(p+v_{\lambda}^{\prime},x,\omega)= 0 &&
		\lambda v_{+,\lambda}+H^{+}(p+v_{+,\lambda}^{\prime},x,\omega)= 0
	\end{eqnarray*}

	By $H^{+}(p,x,\omega)\geq H(p,x,\omega)$, $\esssup\limits\limits_{(x,\omega)}H(p,x,\omega)\geq 0$ and comparison principle, we have
	\begin{eqnarray*}
		\liminf_{\lambda\rightarrow 0}-\lambda v_{+,\lambda}(0,p,\omega)\geq \liminf_{\lambda\rightarrow 0}-\lambda v_{\lambda}(0,p,\omega)\geq 0
	\end{eqnarray*}

	Thus, if $\overline{H^{+}}(p)=0$, then $\overline{H}(p)=0$. Since $\overline{H^{+}}(0)=0$, we can only consider the case: $p>0$ and $\overline{H^{+}}(p)>0$. By Lemma \ref{Two points control}, for a.e. $\omega\in\Omega$, any $R>0$, there exists $\lambda_{0}=\lambda_{0}(R,p,\omega)>0$, such that
	\begin{eqnarray*}
		0<\lambda<\lambda_{0}\implies p+v_{+,\lambda}(x,\pagebreak,\omega)\geq 0 & x\in \left[-\frac{R}{\lambda},\frac{R}{\lambda}\right]
	\end{eqnarray*}

	So $\lambda v_{+,\lambda}+H(p+v_{+,\lambda}^{\prime},x,\omega)= 0, \forall x\in\left[-\frac{R}{\lambda},\frac{R}{\lambda}\right]$. By Lemma \ref{CP}, there is a constant $C>0$ and
	\begin{eqnarray*}
		\left|\lambda v_{+,\lambda}(0,p,\omega)-\lambda v_{\lambda}(0,p,\omega)\right|\leq \frac{C}{R}
	\end{eqnarray*}
	
	Since $R$ can be chosen arbitrarily large,
	\begin{eqnarray*}
		\lim_{\lambda\rightarrow 0}-\lambda v_{\lambda}(0,p,\omega)=\lim_{\lambda\rightarrow 0}-\lambda v_{+,\lambda}(0,p,\omega)=\overline{H^{+}}(p)
	\end{eqnarray*}

	So $\overline{H}(p)=\overline{H^{+}}(p), p\geq 0$. Similarly, we can prove $\overline{H}(p)=\overline{H^{-}}(p), p\leq 0$.

\end{proof}

\begin{remark}
	By Lemma \ref{gluing at minimum point}, to prove the homogenization of a Hamiltonian that satisfies \textbf{(A1)-(A3)} and is constrained with index $(\widetilde{L},L)$, it suffices to study those Hamiltonians that have index $(0,L)$ or $(\widetilde{L},0)$. Without loss of generality, in the following sections, we only consider the Hamiltonian under assumptions \textbf{(A1)-(A3)} and be constrained with index $(0,L)$.
\end{remark}

\section{Gluing Lemmas: reduction from small oscillation to large oscillation}\label{the section of gluing lemmas}

In this section, $H(p,x,\omega)$ satisfies \textbf{(A1)-(A3)} and is constrained with index $(0,L)$. Denote
\begin{eqnarray*}
\underline{M}:=\essinf\limits\limits_{(x,\omega)\in\RR\times\Omega}M(x,\omega) &&
\overline{m}:=\esssup\limits\limits_{(x,\omega)\in\RR\times\Omega}m(x,\omega)
\end{eqnarray*}

 There are $1\leq\underline{k},\overline{k}\leq L$, such that
 \begin{eqnarray*}
 	\underline{M}:=\essinf\limits\limits_{(x,\omega)\in\RR\times\Omega}M_{\underline{k}}(x,\omega) &&
 	\overline{m}:=\esssup\limits\limits_{(x,\omega)\in\RR\times\Omega}m_{\overline{k}}(x,\omega)
 \end{eqnarray*}

\begin{definition}[Oscillation]\label{oscillation}
	Let $H(p,x,\omega)$ be constrained(\ref{constrained Hamiltonian}) and satisfies \textbf{(A1)-(A3)}.
	
	(1)$H(p,x,\omega)$ has small oscillation if $\underline{M}\geq\overline{m}$.
	
	(2)$H(p,x,\omega)$ has large oscillation if $\underline{M}<\overline{m}$.
\end{definition}

Throughout this section, we assume small oscillation and denote
\begin{eqnarray*}
P:= p_{\overline{k}} &&
Q:= q_{\underline{k}}
\end{eqnarray*}

\subsection{Left Steep Side:} $\underline{M}>\overline{m}$ and $P<Q$. Define
\begin{eqnarray*}
H_{1}(p,x,\omega)&:=&\begin{cases}
H(p,x,\omega) & p\leq Q\\
\mathcal{L}|p-Q|+H(Q,x,\omega) & p>Q
\end{cases}\\
H_{3}(p,x,\omega)&:=&\begin{cases}
H(p,x,\omega) & p\geq q_{\bar{k}}\\
\mathcal{L}|p-q_{\bar{k}}|+H(q_{\bar{k}},x,\omega) & p<q_{\bar{k}}
\end{cases}\\
H_{2}(p,x,\omega)&:=&\max\lbrace H_{1}(p,x,\omega), H_{3}(p,x,\omega) \rbrace
\end{eqnarray*}

\begin{figure}[h]
\centering
\includegraphics[width=0.4\linewidth]{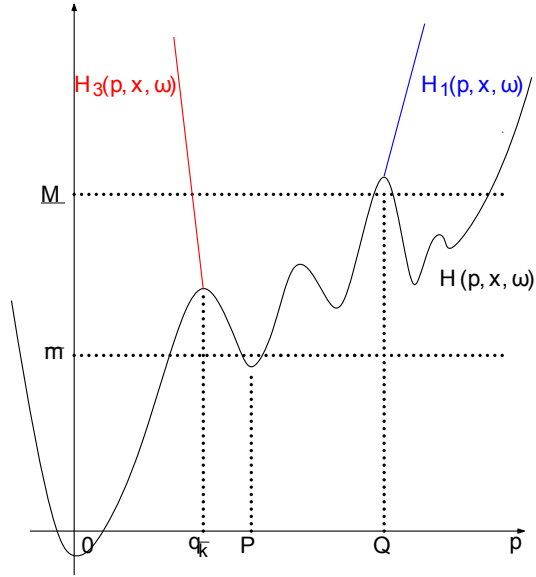}
\caption{Left Steep Side}
\label{fig:leftsteepside}
\end{figure}

\begin{lemma}\label{left steep side}
Assume $H_{i}(p,x,\omega), i=1,2,3$ are all regularly homogenizable for any $p\in\RR$. Then $H(p,x,\omega)$ is also regularly homogenizable for any $p\in\RR$ and 
\begin{eqnarray*}
\overline{H}(p)=\min\left\lbrace \overline{H_{1}}(p), \overline{H_{3}}(p) \right\rbrace
\end{eqnarray*}
\end{lemma}

\begin{proof}[Proof of the periodic case]
	For any $p\in\RR$, we have the cell problem
	\begin{eqnarray}\label{cell problem in periodic case for left steep side}
	H(p+v^{\prime}(x),x)=\overline{H}(p)
	\end{eqnarray}
	
	Proof by contradiction, if there are $x_{1},x_{2}\in[0,1]$, such that $p+v^{\prime}(x_{1})>Q$ and  $p+v^{\prime}(x_{2})<P$. Then $px+v(x)-Qx$ attains local maximum at some $y_{1}\in(x_{1},x_{2}+1)$ and $px+v(x)-Px$ attains local minimum at some $y_{2}\in(x_{2},x_{1}+1)$. Thus we get a contradiction from equalities:
	\begin{eqnarray*}
	\min_{x\in[0,1]} H(Q,x)=\underline{M}\leq H(Q,y_{1})\leq \overline{H}(p)\leq H(P,y_{2})\leq \overline{m}=\max_{x\in[0,1]}H(P,x)
	\end{eqnarray*}

	Thus, either $p+v^{\prime}(x)\leq Q$ for all $x\in[0,1]$ or $p+v^{\prime}(x)\geq P$ for all $x\in[0,1]$. By (\ref{cell problem in periodic case for left steep side}), either $\overline{H}(p)=\overline{H_{1}}(p)$ or $\overline{H}(p)=\overline{H_{3}}(p)$. On the other hand, since $H(p,x,\omega)=\min\lbrace H_{1}(p,x,\omega),H_{3}(p,x,\omega) \rbrace$, by comparison principle, we have $\overline{H}(p)\leq \lbrace \overline{H_{1}}(p),\overline{H_{3}}(p) \rbrace$. Eventually, we conclude
	\begin{eqnarray*}
	\overline{H}(p)= \lbrace \overline{H_{1}}(p),\overline{H_{3}}(p) \rbrace
	\end{eqnarray*} 
\end{proof}

\begin{proof}[Proof of the random case]
Decompose $\RR$ into three parts.

\textbf{(1)}If $p\in(-\infty,P)$, then $\overline{H}(p)=\overline{H_{1}}(p)$.

For each $\omega\in\Omega$ and $\lambda>0$, let $v_{\lambda}(x,p,\omega)$ and $v_{1,\lambda}(x,p,\omega)$ be solutions of the equations
\begin{eqnarray*}
\lambda v_{\lambda}+H(p+v_{\lambda}^{\prime},x,\omega)= 0 &&
\lambda v_{1,\lambda}+H_{1}(p+v_{1,\lambda}^{\prime},x,\omega)= 0
\end{eqnarray*}

By Lemma \ref{Three points control}, there is $\widetilde{\Omega}\subset\Omega$, $\mathbf{P}[\widetilde{\Omega}]=1$. For $\omega\in\widetilde{\Omega}$, any $R>0$, there is $\lambda_{0}=\lambda_{0}(R,p,\omega)>0$,
\begin{eqnarray*}
0<\lambda<\lambda_{0}\implies p+v_{1,\lambda}^{\prime}(x,p,\omega)\leq Q & x\in\left[-\frac{R}{\lambda},\frac{R}{\lambda}\right]
\end{eqnarray*}

Thus, for $0<\lambda<\lambda_{0}(R,p,\omega)$, 
\begin{eqnarray*}
\lambda v_{\lambda}+H(p+v_{\lambda}^{\prime},x,\omega)= 0 ,
\lambda v_{1,\lambda}+H(p+v_{1,\lambda}^{\prime},x,\omega)= 0 & x\in \left[-\frac{R}{\lambda},\frac{R}{\lambda}\right]
\end{eqnarray*}

By Lemma \ref{CP}, there is $C=C(p)$, such that
\begin{eqnarray*}
\left|\lambda v_{\lambda}(0,p,\omega)-\lambda v_{1,\lambda}(0,p,\omega)\right|\leq \frac{C}{R}
\end{eqnarray*}

Since $R$ can be chosen arbitrarily large
\begin{eqnarray*}
\lim_{\lambda\rightarrow 0^{+}}-\lambda v_{\lambda}(0,p,\omega)=\lim_{\lambda\rightarrow 0^{+}}-\lambda v_{1,\lambda}(0,p,\omega)=\overline{H_{1}}(p)
\end{eqnarray*}

Thus $H(p,x,\omega)$ is regularly homogenizable at $p$ and $\overline{H}(p)=\overline{H_{1}}(p)$, $p\in(-\infty,P)$.

\vspace{5mm}

\textbf{(2)}$p\in(Q,\infty)$, then $\overline{H}(p)=\overline{H_{3}}(p)$.

For each $\omega\in\Omega$ and $\lambda>0$, let $v_{\lambda}(x,p,\omega)$ and $v_{3,\lambda}(x,p,\omega)$ be solutions of the equations
\begin{eqnarray*}
\lambda v_{\lambda}+H(p+v_{\lambda}^{\prime},x,\omega)= 0 &
\lambda v_{3,\lambda}+H_{3}(p+v_{3,\lambda}^{\prime},x,\omega)= 0
\end{eqnarray*}

By Lemma \ref{Three points control}, there is $\widetilde{\Omega}\subset\Omega$, $\mathbf{P}[\widetilde{\Omega}]=1$. For $\omega\in\widetilde{\Omega}$, any $R>0$, there exists some $\lambda_{0}=\lambda_{0}(R,\omega,p)>0$,
\begin{eqnarray*}
0<\lambda<\lambda_{0}\implies p+v_{3,\lambda}^{\prime}(x,p,\omega)\geq P & x\in\left[-\frac{R}{\lambda},\frac{R}{\lambda}\right]
\end{eqnarray*}

Thus, for $0<\lambda<\lambda_{0}(R,p,\omega)$,
\begin{eqnarray*}
\lambda v_{\lambda}+H(p+v_{\lambda}^{\prime},x,\omega)= 0 ,
\lambda v_{3,\lambda}+H(p+v_{3,\lambda}^{\prime},x,\omega)= 0 & x\in\left[-\frac{R}{\lambda},\frac{R}{\lambda}\right]
\end{eqnarray*}

By Lemma \ref{CP}, there is $C=C(p)$, such that 
\begin{eqnarray*}
\left|\lambda v_{\lambda}(0,p,\omega)-\lambda v_{3,\lambda}(0,p,\omega)\right|\leq \frac{C}{R}
\end{eqnarray*}

Since $R$ can be chosen arbitrarily large
\begin{eqnarray*}
\lim_{\lambda\rightarrow 0^{+}}-\lambda v_{\lambda}(0,p,\omega)=\lim_{\lambda\rightarrow 0^{+}}-\lambda v_{3,\lambda}(0,p,\omega)=\overline{H_{3}}(p)
\end{eqnarray*}

Thus $H(p,x,\omega)$ is regularly homogenizable at $p$ and $\overline{H}(p)=\overline{H_{3}}(p)$, $p\in(Q,\infty)$.

\vspace{5mm}

\textbf{(3.1)}Denote:
\begin{eqnarray*}
A:=\big\lbrace p\in(P,Q)\big| \overline{m}<\overline{H_{2}}(p)<\underline{M} \big\rbrace
\end{eqnarray*}

Fix any $p\in A$, for any $\lambda>0$, let $v_{\lambda}(x,p,\omega), v_{2,\lambda}(x,p,\omega)$ be solutions of the equations:
\begin{eqnarray*}
\lambda v_{\lambda}+H(p+v_{\lambda}^{\prime},x,\omega)=0  && \lambda v_{2,\lambda}+H_{2}(p+v_{2,\lambda}^{\prime},x,\omega)=0
\end{eqnarray*}

By Lemma \ref{Two points control}, for each $\omega\in\widetilde{\Omega}$, any $R>0$, there is $\lambda_{0}=\lambda_{0}(R,p,\omega)>0$, s.t. 
\begin{eqnarray*}
0<\lambda<\lambda_{0}\implies P\leq p+v_{2,\lambda}^{\prime}(x,\omega)\leq Q & x\in\left[-\frac{R}{\lambda},\frac{R}{\lambda}\right]
\end{eqnarray*}

So
\begin{eqnarray*}
\lambda v_{\lambda}+H(p+v_{\lambda}^{\prime},x,\omega)=0 , \lambda v_{2,\lambda}+H(p+v_{2,\lambda}^{\prime},x,\omega)=0 &  x\in\left[-\frac{R}{\lambda},\frac{R}{\lambda}\right]
\end{eqnarray*}

By Lemma \ref{CP}, there is $C=C(p)$, such that
\begin{eqnarray*}
\left|\lambda v_{\lambda}(0,p,\omega)-\lambda v_{2,\lambda}(0,p,\omega)\right|\leq \frac{C}{R}
\end{eqnarray*}

Since $R$ can be chosen arbitrarily large
\begin{eqnarray*}
\lim_{\lambda\rightarrow 0^{+}}-\lambda v_{\lambda}(0,p,\omega)=\lim_{\lambda\rightarrow 0^{+}}-\lambda v_{2,\lambda}(0,p,\omega)=\overline{H_{2}}(p)
\end{eqnarray*}

Thus $H(p,x,\omega)$ is regularly homogenizable at $p$ and $\overline{H}(p)=\overline{H_{2}}(p)\geq \lbrace \overline{H_{1}}(p),\overline{H_{3}}(p) \rbrace$. On the other hand $\overline{H}(p)\leq\min\lbrace \overline{H_{1}}(p),\overline{H_{3}}(p) \rbrace$. So $\overline{H}(p)=\overline{H_{1}}(p)=\overline{H_{2}}(p)=\overline{H_{3}}(p), p\in A$.

\vspace{5mm}

\textbf{(3.2)} For $p\in\RR$, if $\overline{H_{1}}(p)<\underline{M}$, then $\overline{H}(p)=\overline{H_{1}}(p)$.

The assumption $\overline{H_{1}}(p)<\underline{M}$ implies $p<Q$. By Lemma \ref{Two points control}, for $\omega\in\widetilde{\Omega}$, any $R>0$, there is $\lambda_{0}=\lambda_{0}(R,p,\omega)>0$, such that
\begin{eqnarray*}
0<\lambda<\lambda_{0}\implies p+\lambda v_{1,\lambda}(x,p,\omega)\leq Q &  x\in\left[-\frac{R}{\lambda},\frac{R}{\lambda}\right]
\end{eqnarray*}

So
\begin{eqnarray*}
\lambda v_{\lambda}+H(p+v_{\lambda}^{\prime},x,\omega)=0 , \lambda v_{1,\lambda}+H(p+v_{1,\lambda}^{\prime},x,\omega)=0 & x\in\left[-\frac{R}{\lambda},\frac{R}{\lambda}\right]
\end{eqnarray*}

By Lemma \ref{CP}, there is $C=C(p)$, such that 
\begin{eqnarray*}
\left|\lambda v_{\lambda}(0,p,\omega)-\lambda v_{1,\lambda}(0,p,\omega)\right|\leq \frac{C}{R}
\end{eqnarray*}

Since $R$ can be chosen arbitrarily large
\begin{eqnarray*}
\lim_{\lambda\rightarrow 0^{+}}-\lambda v_{\lambda}(0,p,\omega)=\lim_{\lambda\rightarrow 0^{+}}-\lambda v_{1,\lambda}(0,p,\omega)=\overline{H_{1}}(p)
\end{eqnarray*}

Thus $H(p,x,\omega)$ is regularly homogenizable at $p$ and $\overline{H}(p)=\overline{H_{1}}(p)$.

\vspace{5mm}

\textbf{(3.3)} For $p>P$, if $\overline{H_{3}}(p)>\overline{m}$, then $\overline{H}(p)=\overline{H_{3}}(p)$.

By Lemma \ref{Two points control}, for each $\omega\in\widetilde{\Omega}$, any $R>0$, there is $\lambda_{0}=\lambda_{0}(R,p,\omega)>0$, such that
\begin{eqnarray*}
0<\lambda<\lambda_{0}\implies p+\lambda v_{3,\lambda}(x,p,\omega)\geq P & x\in\left[-\frac{R}{\lambda},\frac{R}{\lambda}\right]
\end{eqnarray*}

So
\begin{eqnarray*}
\lambda v_{\lambda}+H(p+v_{\lambda}^{\prime},x,\omega)=0 , \lambda v_{3,\lambda}+H(p+v_{3,\lambda}^{\prime},x,\omega)=0 & x\in\left[-\frac{R}{\lambda},\frac{R}{\lambda}\right]
\end{eqnarray*}

By Lemma \ref{CP}, there is $C=C(p)$, such that
\begin{equation*}
\left|\lambda v_{\lambda}(0,p,\omega)-\lambda v_{3,\lambda}(0,p,\omega)\right|\leq \frac{C}{R}
\end{equation*}

Since $R$ can be chosen arbitrarily large
\begin{equation*}
\lim_{\lambda\rightarrow 0^{+}}-\lambda v_{\lambda}(0,p,\omega)=\lim_{\lambda\rightarrow 0^{+}}-\lambda v_{3,\lambda}(0,p,\omega)=\overline{H_{3}}(p)
\end{equation*}

Thus $H(p,x,\omega)$ is regularly homogenizable at $p$ and $\overline{H}(p)=\overline{H_{3}}(p)$.\\

\vspace{5mm}

\textbf{(3.4)} For $p<Q$, if $\overline{H_{3}}(p)<\underline{M}$, then $\overline{H_{2}}(p)=\overline{H_{3}}(p)<\underline{M}$.

By Lemma \ref{Two points control}, for each $\omega\in\widetilde{\Omega}$, any $R>0$, there is $\lambda_{0}=\lambda_{0}(R,p,\omega)>0$, such that
\begin{eqnarray*}
0<\lambda<\lambda_{0}\implies p+ v_{3,\lambda}^{\prime}(x,p,\omega,)\leq Q & x\in\left[-\frac{R}{\lambda},\frac{R}{\lambda}\right]
\end{eqnarray*}

Here, for any $\lambda>0$, $v_{3,\lambda}$ is the solution of the equation
\begin{eqnarray*}
\lambda v_{3,\lambda}+H_{3}(p+v_{3,\lambda}^{\prime},x,\omega)=0
\end{eqnarray*}

However, by the above upper bound,
\begin{eqnarray*}
\lambda v_{3,\lambda}+H_{2}(p+v_{3,\lambda}^{\prime},x,\omega)=0 & x\in\left[-\frac{R}{\lambda},\frac{R}{\lambda}\right]
\end{eqnarray*}

Suppose for any $\lambda>0$, $v_{2,\lambda}(x,p,\omega)$ is the solution of the equation:
\begin{eqnarray*}
\lambda v_{2,\lambda}+H_{2}(p+v_{2,\lambda}^{\prime},x,\omega)=0
\end{eqnarray*}

By Lemma \ref{CP}, there is $C=C(p)$, such that
\begin{eqnarray*}
\left|\lambda v_{2,\lambda}(0,p,\omega)-\lambda v_{3,\lambda}(0,p,\omega)\right|\leq \frac{C}{R}
\end{eqnarray*}

Since $R$ can be chosen arbitrarily large
\begin{eqnarray*}
\overline{H_{2}}(p)=\lim_{\lambda\rightarrow 0^{+}}-\lambda v_{2,\lambda}(0,p,\omega)=\lim_{\lambda\rightarrow 0^{+}}-\lambda v_{3,\lambda}(0,p,\omega)=\overline{H_{3}}(p)
\end{eqnarray*}

Now, we discuss the homogenization of $\overline{H}(p)$ for $p\in[P,Q]\cap A^{c}$.

\textbf{(I)} If $p\in(P,Q)$ and $\overline{H_{2}}(p)\leq\overline{m}$, by the fact $\overline{m}<\underline{M}$ and
\begin{eqnarray*}
\max\lbrace \overline{H_{1}}(p), \overline{H_{3}}(p) \rbrace\leq \overline{H_{2}}(p)
\end{eqnarray*}

we have $\overline{H_{1}}(p)<\underline{M}$, by \textbf{(3.2)}, $\overline{H}(p)=\overline{H_{1}}(p)$.

\textbf{(II)} If $p\in(P,Q)$ and $\overline{H_{2}}(p)\geq\underline{M}$, then by \textbf{(3.4)}, $\overline{H_{3}}(p)\geq\underline{M}>\overline{m}$. By \textbf{(3.3)}, $\overline{H}(p)=\overline{H_{3}}(p)$.

\textbf{(III)} By Corollary \ref{closedness of regularly homogenizable points}, we have
\begin{eqnarray*}
\overline{H}(P)=\overline{H_{1}}(P)\text{ and } \overline{H}(Q)=\overline{H_{3}}(Q)
\end{eqnarray*}

In all, for any $p\in\RR$, either $\overline{H}(p)=\overline{H_{1}}(p)$ or $\overline{H}(p)=\overline{H_{3}}(p)$, so
\begin{eqnarray*}
\overline{H}(p)\geq\min\big\lbrace \overline{H_{1}}(p),\overline{H_{3}}(p) \big\rbrace
\end{eqnarray*}

On the other hand
\begin{eqnarray*}
\overline{H}(p)\leq\min\big\lbrace \overline{H_{1}}(p),\overline{H_{3}}(p) \big\rbrace
\end{eqnarray*}

So, we have proved:
\begin{eqnarray*}
\overline{H}(p)=\min\big\lbrace \overline{H_{1}}(p),\overline{H_{3}}(p) \big\rbrace
\end{eqnarray*}
\end{proof}

\subsection{Right Steep Side:}  $\underline{M}>\overline{m}$ and $Q\leq P$. Define
\begin{eqnarray*}
H_{1}(p,x,\omega):=\begin{cases}
H(p,x,\omega) & p\leq Q\\
\mathcal{L}|p-Q|+H(Q,x,\omega) & p> Q
\end{cases}
\end{eqnarray*}

\begin{eqnarray*}
H_{2}(p,x,\omega)=\begin{cases}
-\mathcal{L}|p|+H(0,x,\omega) & p<0\\
H(p,x,\omega) & 0\leq p\leq P\\
-\mathcal{L}|p-P|+H(P,x,\omega) & p>P
\end{cases}
\end{eqnarray*}

\begin{eqnarray*}
H_{3}(p,x,\omega):=\begin{cases}
H(p,x,\omega) & p\geq Q\\
\mathcal{L}|p-Q|+H(Q,x,\omega) & p< Q
\end{cases}
\end{eqnarray*}

\begin{figure}[h]
\centering
\includegraphics[width=0.4\linewidth]{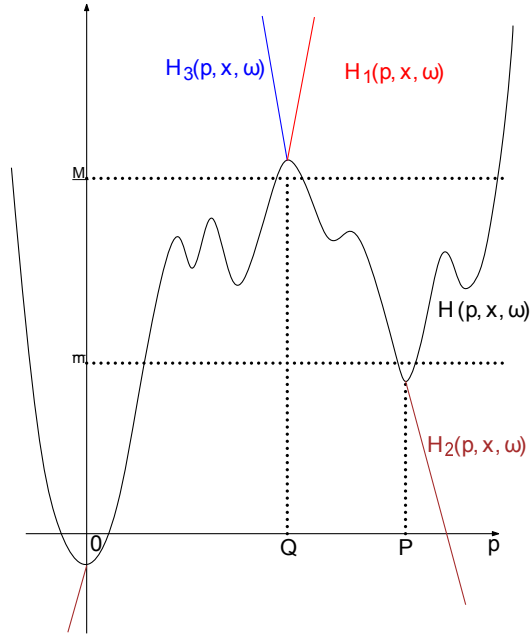}
\caption{Right Steep Side}
\label{fig:RightSteepSide}
\end{figure}

\begin{lemma}\label{right steep side}
Assume both $H_{1}(p,x,\omega)$ and $H_{3}(p,x,\omega)$ are regularly homogenizable for all $p\in\RR$, then $H(p,x,\omega)$ is also regularly homogenizable for all $p$ and
\begin{eqnarray*}
\overline{H}(p)=\begin{cases}
\overline{H_{1}}(p) & p\leq 0\\
\min\lbrace \overline{H_{1}}(p), \overline{H_{3}}(p), \underline{M} \rbrace & p\in(0,P)\\
\overline{H_{3}}(p) & p\geq P\\
\end{cases}
\end{eqnarray*}
\end{lemma}

\begin{proof}[Proof of the periodic case for middle equality] For $p\in(0,P)$, we have the cell problem
	\begin{eqnarray*}
	H(p+v^{\prime}(x),x)=\overline{H}(p)
	\end{eqnarray*}

If $p+v^{\prime}(x)\leq Q,\forall x\in[0,1]$ or $p+v^{\prime}(x)\geq Q,\forall x\in[0,1]$, then $\overline{H}(p)=\overline{H_{1}}(p)$ or $\overline{H}(p)=\overline{H_{3}}(p)$.

Otherwise, by the assumption that $\underline{M}>\overline{m}$, we have $p+v^{\prime}(x)\in[0,P],\forall x\in[0,1]$.

There is some $x_{0}\in[0,1]$, such that $H(Q,x_{0})=\min\limits_{x} \max\limits_{q\in[0,P]}H(q,x)=\underline{M}$. So we have $\overline{H}(p)=H(p+v^{\prime}(x_{0}),x_{0})\leq \underline{M}$. Thus
\begin{eqnarray*}
\overline{H}(p)\leq \min\lbrace \overline{H_{1}}(p), \overline{H_{3}}(p), \underline{M} \rbrace
\end{eqnarray*}

If $\overline{H}(p)<\underline{M}$, then by Lemma \ref{Two points control}, we have either $p+v^{\prime}(x)\leq Q,\forall x\in[0,1]$ or $p+v^{\prime}(x)\geq Q,\forall x\in[0,1]$ and so $\overline{H}(p)=\overline{H_{1}}(p)$ or $\overline{H}(p)=\overline{H_{3}}(p)$.

\end{proof}

\begin{proof}[Proof of random case]
STEP 1: Proof of the first equality. Define 
\begin{eqnarray*}
f(\theta):=\esssup\limits\limits_{(x,\omega)\in\RR\times\Omega}\left[H(\theta Q,x,\omega)\right]
\end{eqnarray*}

Then $f(0)=0, f(1)\geq \underline{M}>\overline{m}>0$.

By the continuity of $f$, there is some $\theta_{0}\in(0,1)$, such that $0<f(\theta_{0})<\underline{M}$.

For any $p\leq 0$, $\lambda>0$, let $v_{1,\lambda}(x,p,\omega)$ be the solution of the equation
\begin{eqnarray*}
\lambda v_{1,\lambda}+H_{1}(p+v_{1,\lambda}^{\prime},x,\omega)=0
\end{eqnarray*}

Apply Lemma \ref{Three points control} to $(p,\theta_{0}Q,Q)$ and $H_{1}(p,x,\omega)$, then for a.e. $\omega\in\Omega$, we have: for any $R>0$, there exists $\lambda_{0}=\lambda_{0}(R,p,\omega)>0$,
\begin{eqnarray*}
0<\lambda<\lambda_{0}\implies p+v_{1,\lambda}^{\prime}(x,p,\omega)\leq Q & x\in\left[-\frac{R}{\lambda},\frac{R}{\lambda}\right]
\end{eqnarray*}

Then by the definition of $H_{1}(p,x,\omega)$, we have 
\begin{eqnarray*}
\lambda v_{1,\lambda}+H(p+v_{1,\lambda}^{\prime},x,\omega)=0 & x\in\left[-\frac{R}{\lambda},\frac{R}{\lambda}\right]
\end{eqnarray*}

For any $\lambda>0$, let $v_{\lambda}$ be the unique viscosity solution of the equation
\begin{eqnarray*}
\lambda v_{\lambda}+H(p+v_{\lambda}^{\prime},x,\omega)=0 & x\in\RR
\end{eqnarray*}

By Lemma \ref{CP}, there is $C=C(p)>0$, such that
\begin{eqnarray*}
|\lambda v_{\lambda}(0,p,\omega)-\lambda v_{1,\lambda}(0,p,\omega)|\leq \frac{C}{R}
\end{eqnarray*}

Since $R$ can be chosen arbitrarily large,
\begin{eqnarray*}
\lim_{\lambda\rightarrow 0}-\lambda v_{\lambda}(0,p,\omega)=\lim_{\lambda\rightarrow 0}-\lambda v_{1,\lambda}(0,p,\omega)=\overline{H_{1}}(p)
\end{eqnarray*}

Thus, $H$ is regularly homogenizable at $p$ and 
\begin{eqnarray*}
\overline{H}(p)=\overline{H_{1}}(p) & p\leq 0
\end{eqnarray*}

STEP 2: Proof of the third equality. Similar as the proof of Step 1.

STEP 3: The second equality.

\textbf{(3.1) Claim:} For $p_{0}\in\RR$, if $\overline{H_{1}}(p_{0})<\underline{M}$, then $H(p,x,
\omega)$ is regularly homogenizable at $p_{0}$ and
$\overline{H}(p_{0})=\overline{H_{1}}(p_{0})$.

\begin{proof}[Proof of \textbf{(3.1) Claim}]
By the definition of $H_{1}(p,x,\omega)$,  $\overline{H_{1}}(p_{0})<\underline{M}$ implies $p<Q$(since $\overline{H_{1}}(p)\geq \underline{M}$ for $p\geq Q$). For each $\omega\in\Omega$ and $\lambda>0$, let $v_{\lambda}(x,p_{0},\omega)$ and $v_{1,\lambda}(x,p_{0},\omega)$ be solutions of the equations
\begin{eqnarray*}
\lambda v_{\lambda}+H(p_{0}+v_{\lambda}^{\prime},x,\omega)=0 &&
\lambda v_{1,\lambda}+H_{1}(p_{0}+v_{1,\lambda}^{\prime},x,\omega)=0
\end{eqnarray*} 

By Lemma \ref{Two points control}, for a.e. $\omega\in\Omega$, we have: for each $R>0$, there is $\lambda_{1}=\lambda_{1}(R,p_{0},\omega)>0$, such that
\begin{eqnarray*}
0<\lambda<\lambda_{0}\implies p_{0}+v_{1,\lambda}^{\prime}\leq Q & x\in \left[-\frac{R}{\lambda},\frac{R}{\lambda}\right]
\end{eqnarray*}

So
\begin{eqnarray*}
\lambda v_{1,\lambda}+H(p_{0}+v_{1,\lambda}^{\prime},x,\omega)=0 & x\in \left[-\frac{R}{\lambda},\frac{R}{\lambda}\right]
\end{eqnarray*}

By Lemma \ref{CP}, there is $C=C(p_{0})>0$, such that
\begin{eqnarray*}
|\lambda v_{\lambda}(0,p_{0},\omega)-\lambda v_{1,\lambda}(0,p_{0},\omega)|<\frac{C}{R}
\end{eqnarray*}

Since we can choose arbitrary large $R$, we have that
\begin{eqnarray*}
\lim_{\lambda\rightarrow 0}-\lambda v_{\lambda}(0,p_{0},\omega=\lim_{\lambda\rightarrow 0}-\lambda v_{1,\lambda}(0,p_{0},\omega=\overline{H_{1}}(p_{0})
\end{eqnarray*}

Thus $H(p,x,\omega)$ is regularly homogenizable at $p_{0}$ and $\overline{H}(p_{0})=\overline{H_{1}}(p_{0})$.
\end{proof}

\textbf{(3.2) Claim:} For $p_{0}\in\RR$, if $\overline{H_{3}}(p_{0})<\underline{M}$, then $H(p,x,
\omega)$ is regularly homogenizable at $p_{0}$ and $\overline{H}(p_{0})=\overline{H_{3}}(p_{0})$.

\begin{proof}[Proof of \textbf{(3.2) Claim}]
The proof is similar as \textbf{(3.1) Claim}.
\end{proof}

\textbf{(3.3)} Denote
\begin{eqnarray*}
q_{1}=\min\big\lbrace p\in[0,P]\big|\overline{H_{1}}(p)=\underline{M} \big\rbrace & q_{2}=\max\big\lbrace p\in[0,P]\big|\overline{H_{3}}(p)=\underline{M} \big\rbrace
\end{eqnarray*}

\textbf{(3.1)}, \textbf{(3.2)} $\implies$ $H(p,x,\omega)$ is regularly homogenizable for $p\in(0,q_{1})\bigcup(q_{2},P)$ and 
\begin{eqnarray*}
\overline{H}(p)=\begin{cases}
\overline{H_{1}}(p) & p\in(0,q_{1})\\
\overline{H_{3}}(p) & p\in(q_{2},P)
\end{cases}
\end{eqnarray*}

By Corollary \ref{closedness of regularly homogenizable points}, $H(p,x,\omega)$ is regularly homogenizable at $q_{1}$ and $q_{2}$ and 
\begin{eqnarray*}
\overline{H}(q_{1})=\overline{H}(q_{2})=\underline{M}
\end{eqnarray*}

\textbf{(3.4) Claim:} $H_{2}(p,x,\omega)$ is regularly homogenizable at $q_{1}$ and $q_{2}$, moreover,
\begin{eqnarray*}
\overline{H_{2}}(q_{1})=\overline{H_{2}}(q_{2})=\underline{M}
\end{eqnarray*}

\begin{proof}[Proof of \textbf{(3.4) Claim}]
By the definition, we have $q_{1}, q_{2}\in(0,P_{0})$. For any $\omega\in\Omega$, $\lambda>0$, let $v_{\lambda}(x,q_{i},\omega)$ and $v_{2,\lambda}(x,q_{i},\omega)$ be solutions to the following equations
\begin{eqnarray*}
\lambda v_{\lambda}+H(q_{i}+v_{\lambda}^{\prime},x,\omega)=0 &&
\lambda v_{2,\lambda}+H_{2}(q_{i}+v_{2,\lambda}^{\prime},x,\omega)=0
\end{eqnarray*}

By the fact that
\begin{eqnarray*}
\overline{H}(q_{i})=\underline{M}>\max\lbrace \esssup\limits\limits_{(x,\omega)}H(0,x,\omega), H(P,x,\omega) \rbrace=\overline{m}
\end{eqnarray*}

By Lemma \ref{Two points control}, then for a.e. $\omega\in\Omega$, for any $R>0$, there is $\lambda_{2}=\lambda_{2}(q_{i},R,\omega)>0$,
\begin{eqnarray*}
0<\lambda<\lambda_{2}\implies 0\leq q_{i}+v_{\lambda}(x,q_{i},\omega)\leq Q & x\in\left[-\frac{R}{\lambda},\frac{R}{\lambda}\right]
\end{eqnarray*}

So we have
\begin{eqnarray*}
\lambda v_{\lambda}+H_{2}(q_{i}+v_{\lambda}^{\prime},x,\omega)=0 & x\in\left[-\frac{R}{\lambda},\frac{R}{\lambda}\right]
\end{eqnarray*}

Apply Lemma \ref{CP}, there is some constant $C=C(q_{i})>0$, such that
\begin{eqnarray*}
|\lambda v_{\lambda}(0,q_{i},\omega)-\lambda v_{2,\lambda}(0,q_{i},\omega)|<\frac{C}{R}
\end{eqnarray*}

We can choose arbitrary large $R$, so
\begin{eqnarray*}
\lim_{\lambda\rightarrow 0}-\lambda v_{2,\lambda}(0,q_{i},\omega)=\lim_{\lambda\rightarrow 0}-\lambda v_{\lambda}(0,q_{i},\omega)=\overline{H}(q_{i})=\underline{M}
\end{eqnarray*}

Thus, $H_{2}(p,x,\omega)$ is regularly homogenizable at $q_{i}$ and $\overline{H_{2}}(q_{i})=\underline{M}$.
\end{proof}

\textbf{(3.5)}
Denote
\begin{eqnarray*}
\widehat{M}(x,\omega):=\max_{\substack{\overline{k}\leq j \leq L\\
j\ne \underline{k}}}M_{j}(x,\omega)
\end{eqnarray*}

Then we have
\begin{eqnarray*}
\underline{\widehat{M}}:=\essinf\limits\limits_{(x,\omega)\in\RR\times\Omega}\widehat{M}(x,\omega)\leq \underline{M}
\end{eqnarray*}

By Lemma \ref{Stability}, without loss of generality, we can further assume $\underline{\widehat{M}}<\underline{M}$. This means that $\mathbf{E}[\widehat{M}(x,\omega)<\underline{M}]>0$. Denote $\widetilde{H}(p,x,\omega):=-H_{2}(q_{\underline{k},0}-p,x,\omega)+\underline{M}$.

If $w_{\lambda}(x,p,\omega)$ is a viscosity solution to 
\begin{eqnarray*}
\lambda w_{\lambda}+H_{2}(p+w_{\lambda}^{\prime},x,\omega)=0
\end{eqnarray*}

Then $\widetilde{w}_{\lambda}(x,p,\omega):=-w_{\lambda}(x,p,\omega)$ is a viscosity solution to 
\begin{eqnarray*}
\lambda \widetilde{w}_{\lambda}+\widetilde{H}(q_{\underline{k},0}-p+\widetilde{w}_{\lambda}^{\prime},x,\omega)+\underline{M}=0
\end{eqnarray*}

Apply Lemma \ref{squeeze lemma} to $\widetilde{H}(p,x,\omega)$, we deduce that $\overline{H_{2}}|_{[q_{1},q_{2}]}\equiv \underline{M}$.

\textbf{(3.6)}For each $p\in[q_{1},q_{2}]$, let $v_{\lambda}(x,p,\omega)$ be the solution to 
\begin{eqnarray*}
\lambda v_{\lambda}(x,p,\omega)+H(p+v_{\lambda}^{\prime}(x,p,\omega),x,\omega)=0
\end{eqnarray*}

By that fact that $H(p,x,\omega)\geq H_{2}(p,x,\omega)$, we have
\begin{eqnarray*}
\mathbf{E}[\omega\in\Omega|\liminf_{\lambda\rightarrow 0}-\lambda v_{\lambda}(0,p,\omega)\geq\underline{M}]=1
\end{eqnarray*}

We only need to show
\begin{eqnarray*}
\mathbf{E}[\omega\in\Omega|\limsup_{\lambda\rightarrow 0}-\lambda v_{\lambda}(0,p,\omega)\leq\underline{M}]=1
\end{eqnarray*}

\textbf{(3.7)} Define $\widehat{H}_{2}(p,x,\omega)$ as following: 
\begin{eqnarray*}
\widehat{H}_{2}(p,x,\omega)=\begin{cases}
H_{2}(p,x,\omega) & p\in(-\infty,0)\cup (P,\infty)\\
\text{concave envelope of }H_{2}(p,x,\omega)|_{p\in[0,P]} & p\in[0,P]\\
\end{cases}
\end{eqnarray*}

By definition, $\widehat{H}_{2}(p,x,\omega)$ is determined by those stationary functions: $m_{i}(x,\omega), M_{j}(x,\omega)$, $1\leq i,j\leq L$, so $\widehat{H}_{2}(p,x,\omega)$ is stationary. Then by the theory of level-set convex homogenization, $\widehat{H}_{2}(p,x,\omega)$ can be homogenized to some level-set concave effective Hamiltonian $\overline{\widehat{H}}_{2}(p)\leq \underline{M}$. 

Since $\widehat{H}_{2}(p,x,\omega)\geq H_{2}(p,x,\omega)$, there exists $\widehat{q}_{1}<\widehat{q}_{2}$ such that $[q_{1},q_{2}]\subset[\widehat{q}_{1},\widehat{q}_{2}]$ and $\overline{\widehat{H}}_{2}(\widehat{q}_{1})=\overline{\widehat{H}}_{2}(\widehat{q}_{2})=\underline{M}$. By level-set concavity, $\overline{\widehat{H}}_{2}|_{[\widehat{q}_{1},\widehat{q}_{2}]}=\underline{M}.$

Denote
\begin{eqnarray*} \widetilde{H}_{2}(p,x,\omega)=\min\lbrace \widehat{H}_{2}(p,x,\omega),\underline{M} \rbrace
\end{eqnarray*}

Then $\widetilde{H}_{2}(p,x,\omega)$ has a level-set concave effective Hamiltonian $\overline{\widetilde{H}}_{2}(p)$ with
\begin{eqnarray*}
\overline{\widetilde{H}}_{2}|_{[\widehat{q}_{1},\widehat{q}_{2}]}=\underline{M}
\end{eqnarray*}

For any $p_{1}\in[\widehat{q}_{1},\widehat{q}_{2}]$ and $\lambda>0$, let $\widehat{v}_{\lambda}(x,p_{1},\omega)$ be the solution of the equation
\begin{eqnarray*}
\lambda \widehat{v}_{2,\lambda}+\widehat{H}_{2}(p_{1}+\widehat{v}_{2,\lambda}^{\prime},x,\omega)=0
\end{eqnarray*}

We will have
\begin{eqnarray*}
\lim_{\lambda\rightarrow 0}\inf_{x\in B_{\frac{R}{\lambda}}}-\lambda \widehat{v}_{2,\lambda}(x,p_{1},\omega)\geq \underline{M}
\end{eqnarray*}

Since $p_{1}<P$ and $0<\overline{m}<\underline{M}$, by Lemma \ref{Two points control}, we have that: for a.e. $\omega\in\Omega$, any $R>0$, there is some $\lambda_{0}=\lambda_{0}(R,p_{1},\omega)>0$, when $\lambda<\lambda_{0}$,
\begin{eqnarray*}
0\leq p_{1}+\widehat{v}_{2,\lambda}^{\prime}(x,p_{1},\omega)\leq P& x\in \left[-\frac{R}{\lambda},\frac{R}{\lambda}\right]
\end{eqnarray*}

Define
\begin{eqnarray*}
\widehat{H}(p,x,\omega)=\begin{cases}
H(p,x,\omega) & p\in (-\infty,0)\cup (P,\infty)\\
\widehat{H}_{2}(p,x,\omega) & p\in[0,P]\\
\end{cases}
\end{eqnarray*}

For each $\omega\in\Omega$ and $\lambda>0$, let $\widehat{v}_{\lambda}(x,p_{1},\omega)$ be the solution of the equation
\begin{eqnarray*}
\lambda \widehat{v}_{\lambda}+\widehat{H}(p_{1}+\widehat{v}_{\lambda}^{\prime},x,\omega)=0
\end{eqnarray*}

Thus
\begin{eqnarray*}
\lambda \widehat{v}_{2,\lambda}+\widehat{H}(p_{1}+\widehat{v}_{2,\lambda}^{\prime},x,\omega)=0 & x\in \left[-\frac{R}{\lambda},\frac{R}{\lambda}\right]
\end{eqnarray*}

By Lemma \ref{CP}, there is some constant $C=C(p_{1})>0$, such that
\begin{eqnarray*}
\Big|\lambda \widehat{v}_{\lambda}(0,p_{1},\omega)-\lambda \widehat{v}_{2,\lambda}(0,p_{1},\omega)\Big|\leq \frac{C}{R}
\end{eqnarray*}

We can choose arbitrary large $R$, so
\begin{equation*}
\lim_{\lambda\rightarrow 0}-\lambda \widehat{v}_{\lambda}(0,p_{1},\omega)=\lim_{\lambda\rightarrow 0}-\lambda \widehat{v}_{2,\lambda}(0,p_{1},\omega)=\underline{M}
\end{equation*}

This means that
\begin{equation*}
\overline{\widehat{H}}|_{[q_{1},q_{2}]}\equiv \underline{M}
\end{equation*}

By the fact that $\widehat{H}(p,x,\omega)\geq H(p,x,\omega)$, we have

\begin{equation*}
\mathbf{E}[\omega\in\Omega|\limsup_{\lambda\rightarrow 0}-\lambda v_{\lambda}(0,p,\omega)\leq\underline{M}]=1
\end{equation*}

This completes the proof.

\end{proof}

\begin{lemma}\label{approx M=m }
Let $H(p,x,\omega)$ be constrained Hamiltonian that satisfies \textbf{(A1)-(A3)} and $\underline{M}=\overline{m}$, then there is a family of Hamiltonians $\lbrace H_{n}(p,x,\omega) \rbrace_{n\in\NN}$, each $H_{n}(p,x,\omega)$ is a constrained Hamiltonian and satisfies \textbf{(A1)-(A3)}, moreover, we have $\underline{M_{n}}>\overline{m_{n}}$ and
\begin{eqnarray*}
\lVert H_{n}(p,x,\omega)-H(p,x,\omega) \rVert_{L^{\infty}(\RR\times\RR\times\Omega)}\leq \frac{1}{n}
\end{eqnarray*}
\end{lemma}

\begin{proof}
For each $n\in\NN$, define the function
\begin{eqnarray*}
h_{n}(p,x,\omega):=\begin{cases}
\frac{p-q_{\underline{k}}}{n(p_{\overline{k}}-q_{\underline{k}})} & p\in[q_{\underline{k}},p_{\overline{k}}]\\
-\frac{p-p_{\overline{k}}}{n(q_{\underline{k}-1}-p_{\overline{k}})}+\frac{1}{n} & p\in(p_{\overline{k}},q_{\underline{k}-1})\\
0 & \text{elsewhere}\\
\end{cases}
\end{eqnarray*}

And define
\begin{eqnarray*}
H_{n}(p,x,\omega):=H(p,x,\omega)-h_{n}(p,x,\omega)
\end{eqnarray*}

Since $q_{\underline{k}}(x,\omega)$, $p_{\overline{k}}(x,\omega)$ and $q_{\underline{k}-1}(x,\omega)$ are all stationary, $H_{n}(p,x,\omega)$ is also stationary. By the construction, we have
\begin{eqnarray*}
\overline{m_{n}}=\overline{m}-\frac{1}{n}=\underline{M}-\frac{1}{n}<\underline{M_{n}}-\frac{1}{n}<\underline{M_{n}}
\end{eqnarray*}

Moreover
\begin{eqnarray*}
\lVert H_{n}(p,x,\omega)-H(p,x,\omega) \rVert_{L^{\infty}(\RR\times\RR\times\Omega)}= \frac{1}{n}
\end{eqnarray*}

\end{proof}

\begin{remark}
	In Lemma \ref{approx M=m }, if those $H_{n}(p,x,\omega)$ are regularly homogenizable for all $p\in\RR$, then according to Lemma \ref{Stability}, $H(p,x,\omega)$ is also regularly homogenizable and $\overline{H}(p)=\lim\limits_{n\rightarrow \infty}\overline{H_{n}}(p).$
\end{remark}

\begin{remark}
The point of Lemma \ref{left steep side}, Lemma \ref{right steep side} and Lemma \ref{approx M=m } is: to prove the homogenization of constrained Hamiltonian $H(p,x,\omega)$ with index $(0,L)$ and has small oscillation, it suffices to study the homogenization of constrained Hamiltonian $H(p,x,\omega)$ with index $(0,L)$ and has large oscillation. 
\end{remark}

\section{Auxiliary Lemmas for Large Oscillation}\label{Auxiliary Lemmas for Large Oscillation}

\subsection{Existence Lemma}
\begin{lemma}\label{existence lemma}
Let Hamiltonian $H(p,x,\omega)$ satisfy \textbf{(A1)-(A3)} and be constrained with index $(0,L)$, then for any $\mu\geq 0, \omega\in\Omega$, there is a Lipschitz continuous viscosity solution $u(x,\omega)$ to the equation:
\begin{eqnarray*}
\begin{cases}
H(u^{\prime},x,\omega)=\mu\\
u^{\prime}\geq 0
\end{cases} && x\in\RR
\end{eqnarray*}
\end{lemma}

\begin{proof}
Fix $\mu\geq 0$ and $\omega\in\Omega$. By \textbf{(A2)}, there exists $p_{0}>0$, such that $H(p_{0},x,\omega)> \mu$. Since $H(0,x,\omega)\leq \mu$, $u_{+}:=p_{0}x$ is a super-solution and $u_{-}:=C$ is a sub-solution for any constant $C$.

\vspace{3mm}
STEP 1. Fix $a\in\RR$ and let $C_{a}:=p_{0}a$, then
\begin{eqnarray*}
u_{+}(a,\omega)=u_{-}(a,\omega)\text{ and }u_{+}(x,\omega)>u_{-}(x,\omega) & \forall x\in(a,\infty)
\end{eqnarray*}

Define
\begin{eqnarray*}
u_{a}(x,\omega)&:=&\sup_{v}\left\lbrace v(x,\omega)\in C([a,\infty))| H(v^{\prime},x,\omega)\leq \mu, C_{a}\leq v(x,\omega)\leq p_{0}x \right\rbrace
\end{eqnarray*}

Then
\begin{eqnarray*}
\begin{cases}
H(u_{a}^{\prime},x,\omega)=\mu & x\in(a,\infty)\\
u_{a}(a,\omega)=p_{0}a
\end{cases}
\end{eqnarray*}

\vspace{3mm}
STEP 2. Fix any $a<b$, denote
\begin{eqnarray*}
w(x,\omega):=u_{a}(x,\omega)+[u_{b}(b,\omega)-u_{a}(b,\omega)] && x\geq b
\end{eqnarray*}

Then
\begin{eqnarray*}
\begin{cases}
H(w^{\prime},x,\omega)=\mu & x\in(b,\infty)\\
w(b,\omega)=p_{0}b
\end{cases}
\end{eqnarray*}

So $u_{b}(x,\omega)\geq w(x,\omega)$ on $[b,\infty)$. Denote
\begin{eqnarray*}
\widetilde{u}_{a}(x,\omega):=\begin{cases}
u_{a}(x,\omega) & x\in[a,b]\\
u_{b}(x,\omega)-u_{b}(b,\omega)+u_{a}(b,\omega) & x\in(b,\infty)\\
\end{cases}
\end{eqnarray*}

Then
\begin{eqnarray*}
p_{0}x\geq\widetilde{u}_{a}(x,\omega)\geq u_{a}(x,\omega)\geq C_{a}, \ x\in [a,\infty)
\end{eqnarray*}

On the other hand, by the construction, $\widetilde{u}_{a}(x,\omega)$ is a sub-solution, so $\widetilde{u}_{a}(x,\omega)\leq u_{a}(x,\omega)$. 

Thus $\widetilde{u}_{a}(x,\omega)\equiv u_{a}(x,\omega)$, which means
\begin{eqnarray*}
\left(u_{b}(x,\omega)-u_{a}(x,\omega)\right)|_{(b,\infty)}\equiv u_{b}(b,\omega)-u_{a}(b,\omega)
\end{eqnarray*}

The above equality is true for any $a<b$, this also implies $u_{a}^{\prime}(x,\omega)\geq 0$.

\vspace{3mm}

STEP 3. For any $n\in\ZZ$, then
\begin{eqnarray*}
u_{n}(x,\omega)-u_{n}(0,\omega)=u_{n+1}(x,\omega)-u_{n+1}(0,\omega) && \forall x\geq n+1
\end{eqnarray*}

For any $x\in\RR$ let $m:=[x]$ and define
\begin{eqnarray*}
u(x,\omega):=u_{m}(x,\omega)-u_{m}(0,\omega)
\end{eqnarray*}

So $u(x,\omega)$ is a well defined Lipschitz function on $\RR$ and it is the solution of the equation
\begin{eqnarray*}
\begin{cases}
H(u^{\prime},x,\omega)=\mu \\
u^{\prime}\geq 0
\end{cases} && x\in\RR
\end{eqnarray*}

\end{proof}

\subsection{Decomposition Lemma}

\begin{lemma}\label{decompsition lemma}
Let $H(p,x,\omega)$ satisfy \textbf{(A1)-(A3)} and be constrained with index $(0,L)$.
Let $u$ be a Lipschitz continuous viscosity solution of the equation
\begin{eqnarray*}
\begin{cases}
H(u^{\prime}(x,\omega),x,\omega)=\mu\geq 0 &\\
u^{\prime}(x,\omega)\geq 0
\end{cases}
x\in\RR
\end{eqnarray*}
Then there is a sequence $\lbrace b_{i} \rbrace_{i\in\ZZ}$, such that
\begin{eqnarray*}
\lim\limits_{i\rightarrow \pm \infty}b_{i}=\pm \infty, u\in C^{1}(I_{i}), I_{i}=(b_{i},b_{i+1})
\end{eqnarray*}
\begin{eqnarray*}
u^{\prime}(x,\omega)|_{I_{i}}=\psi_{k_{i},(x,\omega)}(\mu) \text{ for some } k_{i}\in\lbrace 1,2,\cdots,2L+1 \rbrace
\end{eqnarray*}
\end{lemma}

\begin{proof}
Fix $\omega\in\Omega$ and omit the notation $\omega$.

\vspace{3mm}

STEP 1. \textbf{Claim}: for each $x\in\RR$, there exist $\delta_{x}>0$ and $l_{x},r_{x}\in\lbrace 1,2,\cdots,2L+1 \rbrace$, such that
\begin{eqnarray*}
u^{\prime}(y)=\begin{cases}
\psi_{l_{x},y}(\mu) & y\in(x-\delta_{x},x)\\
\psi_{r_{x},y}(\mu) & y\in(x,x+\delta_{x})\\
\end{cases}
\end{eqnarray*} 

Just prove the first equality, since the proof for the second one is similar. Suppose this is not true at some $x_{0}$, then there exist two sequences $x_{n}\rightarrow x_{0}^{-}$ and $y_{n}\rightarrow x_{0}^{-}$, $1\leq k_{2}<k_{1}\leq 2L+1$, such that
\begin{eqnarray*}
x_{1}<y_{1}<x_{2}<y_{2}<\cdots<x_{0} &
u^{\prime}(x_{n})=\psi_{k_{1},x_{n}}(\mu) &
u^{\prime}(y_{n})=\psi_{k_{2},y_{n}}(\mu)
\end{eqnarray*}

\textbf{Case 1:} $k_{1}\geq k_{2}+2.$ 
Then there is a branch between the $k_{1}$-th branch and the $k_{2}$-th branch. So there exist $a<b$, such that $u^{\prime}(x_{n})<a<b<u^{\prime}(y_{n})$.

Fix any $p\in [a,b]$, then $u(x)-px$ is decreasing(increasing) around $x_{n}(y_{n})$. So, $u(x)$ attains local minimum(maximum) at $z_{n}^{-}\in(x_{n},y_{n})(z_{n}^{+}\in(y_{n},x_{n+1}))$, then $H(p,z_{n}^{+})\leq \mu\leq H(p,z_{n}^{-})$, and thus there is $z_{n}\in[z_{n}^{-},z_{n}^{+}]$ with $H(p,z_{n})=\mu$. By the fact $\lim\limits_{n\rightarrow \infty}z_{n}=x_{0}$, we have $H(p,x_{0})=\mu.$ This is true for any $p\in[a,b]$ and this contradicts to the fact that $H(p,x,\omega)$ is constrained.

\vspace{3mm}

\textbf{Case 2:} $k_{1}=k_{2}+1,$ without loss of generality, let $k_{1}=2,k_{2}=1$. 

If $m_{1}(x_{0})<\mu$, by the similar argument used in \textbf{Case 1}, we get a contradiction.

If $m_{1}(x_{0})>\mu$, there is some $\delta>0$, s.t. $m_{1}(\cdot)|_{(x_{0}-\delta,x_{0})}>\mu$, let $x_{n}\in(x_{0}-\delta,x_{0})$, then $\mu=H(u^{\prime}(x_{n}),x_{n})\geq H(p_{1},x_{n})>\mu$, which is a contradiction.

If $m_{1}(x_{0})=\mu$, since $m_{1}(x)$ has no cluster point, there is some $\delta>0$ such that $\mu\notin \lbrace m_{1}(x)|x\in(x_{0}-\delta,x_{0})\rbrace$. By the above discussion, $m_{1}(\cdot)|_{(x_{0}-\delta,x_{0})}<\mu$. Let $\Phi(x):=u(x)-p_{1}x$, then $\Phi^{\prime}(x_{n})<0$ and $\Phi^{\prime}(y_{n})>0$, so there is some $z_{n}\in(x_{n},y_{n})$ where $\Phi(x)$ attains local minimum. So $m_{1}(z_{n})=H(p_{1},z_{n})\geq \mu$, since $z_{n}\in(x_{0}-\delta,x_{0})$ when $n\gg 1$, we get the contradiction.

Thus, the \textbf{Claim} is proved.\\

STEP 2. Denote:
$A:=\lbrace x\in\RR |l_{x}\ne r_{x} \rbrace$. By the above arguments, we see that $A$ has no cluster point. Then there is a sequence $\lbrace b_{i} \rbrace_{i\in\ZZ}$ such that $b_{i}<b_{i+1}$, $A\subset\lbrace b_{i} \rbrace_{i\in\ZZ}$ and $\lim\limits_{i\rightarrow \pm \infty}b_{i}=\pm \infty.$ We will have $r_{b_{i}}=l_{b_{i+1}}$. Thus
$u^{\prime}(x)=\psi_{r_{b_{i}},x}(\mu),  x\in(b_{i},b_{i+1})$

\end{proof}

\subsection{Homotopy between solutions}
Let $H(p,x,\omega)$ be constrained with index $(0,L)$, for simplicity of notation, we omit the dependence of $\omega$.
Let $f\in L^{\infty}(\RR)$ and any solution of $u^{\prime}(x)=f(x)$ is a viscosity solution to
\begin{eqnarray*}
\begin{cases}
	H(u^{\prime},x)=f(x)\\
	u^{\prime}\geq 0
\end{cases}
&& x\in\RR
\end{eqnarray*}

By Lemma\ref{decompsition lemma}, let $a_{1}<a_{2}<a_{3}$ and $f(x)|_{(a_{i},a_{i+1})}=\psi_{k_{i},x}(\mu)$, $k_{i}\in\lbrace 1,2,\cdots, 2L+1 \rbrace, i=1,2$. Denote $k=\min\lbrace k_{1},k_{2} \rbrace$ and  define
\begin{eqnarray*}
\widetilde{f}(x):=\begin{cases}
f(x) & x\in\RR\diagdown (a_{1},a_{3})\\
\psi_{k,x}(\mu) & x\in(a_{1},a_{3})\\
\end{cases}
\end{eqnarray*}

\begin{lemma}\label{combination lemma}
Assume $\mu\notin \lbrace m_{i}(x), M_{j}(x)|1\leq i,j\leq L, x\in(a_{1},a_{3}) \rbrace$. Then any solution of $u^{\prime}=\widetilde{f}$ is also a viscosity solution of
\begin{eqnarray*}
H(u^{\prime}(x),x)=\mu & x\in\RR
\end{eqnarray*}
\end{lemma}

\begin{proof}
Similar to the proof of A.3 in [\ref{1-d seperable noncovex by ATY}].
\end{proof}

\vspace{3mm}

Let $I=(a,b)$, and $f_{1},f_{2}\in L^{\infty}(I)$, $f_{1}\geq f_{2}$. Assume solutions of
\begin{eqnarray*}
\begin{cases}
u_{1}^{\prime}=f_{1} & x\in I\\
u_{1}(a)=0 & \\
\end{cases}
&\text{ and }&\begin{cases}
u_{2}^{\prime}=f_{2} & x\in I\\
u_{2}(a)=0 & \\
\end{cases}
\end{eqnarray*}

are both viscosity solutions of the equation
\begin{eqnarray}\label{eqn for homotopy}
H(u^{\prime},x,\omega)=\mu && x\in I
\end{eqnarray}

Then $u_{2}(x)\leq u_{1}(x)\leq u_{2}(x)-u_{2}(b)+u_{1}(b)$. Fix any $c\in[u_{2}(b),u_{1}(b)]$ and define
\begin{eqnarray*}
u_{c,*}(x)&:=&\max \lbrace u_{2}(x),u_{1}(x)-u_{1}(b)+c \rbrace\\
u^{c,*}(x)&:=&\min \lbrace u_{1}(x),u_{2}(x)-u_{2}(b)+c \rbrace
\end{eqnarray*}

Define the set
\begin{eqnarray*}
\mathcal{W}:=\left\lbrace w\in  W^{1,\infty}(I)|H(w^{\prime},x,\omega)\leq \mu \text{ and }u_{c,*}(x)\leq w(x)\leq u^{c,*}(x) \right\rbrace
\end{eqnarray*}

And the function $w_{c}(x):=\sup\limits_{w\in \mathcal{W}} w(x)$. Denote
\begin{eqnarray*}
\mathcal{F}_{I}(f_{1},f_{2},c)(x):=\begin{cases}
w_{c}^{\prime}(x) & \text{ if } w_{c} \text{ is differentiable at } x\\
0 & \text{ otherwise }\\
\end{cases}
\end{eqnarray*}

Then $u_{c,*}(x)(u^{c,*}(x))$ is a viscosity sub(super)solution to equation (\ref{eqn for homotopy}). By Perron's method, $w_{c}(x)$ is a viscosity solution of the equation
\begin{eqnarray*}
\begin{cases}
H(w_{c}^{\prime}(x),x)=\mu \\
w_{c}(a)=0, w_{c}(b)=c
\end{cases}
&& x\in (a,b)
\end{eqnarray*}

\begin{lemma}\label{homotopy between sol}
Fix $a<b$, $0<\epsilon<\frac{b-a}{2}$, let $f_{1},f_{2}\in L^{\infty}(a-\epsilon,b+\epsilon)$ such that
\begin{eqnarray*}
f_{1}(x)\geq f_{2}(x), \  x\in(a-\epsilon,b+\epsilon) && f_{1}(x)=f_{2}(x),\  x\in(a-\epsilon,a)\bigcup(b,b+\epsilon)
\end{eqnarray*}
Suppose any solution of
$\begin{cases}
u_{i}^{\prime}(x)=f_{i}(x) & x\in(a-\epsilon,b+\epsilon)\\
u_{i}(a)=0
\end{cases} (i=1,2)$
is a viscosity (sub-)solution of the equation: $H(u^{\prime},x)=\mu$. Fix $c\in[u_{2}(b),u_{1}(b)]$, then any solution of the equation
\begin{eqnarray*}
v^{\prime}(x)=\begin{cases}
f_{1}(x)=f_{2}(x) & x\in(a-\epsilon,a)\bigcup(b,b+\epsilon)\\
\mathcal{F}_{I}(f_{1},f_{2},c)(x) & x\in I=(a,b)\\
\end{cases}
\end{eqnarray*}
is a viscosity (sub-)solution of the equation
\begin{eqnarray*}
\begin{cases}
H(u^{\prime}(x),x)=\mu & x\in (a-\epsilon,b+\epsilon)\\
u(b)=c & \\
\end{cases}
\end{eqnarray*} 
\end{lemma}

\begin{proof}
Same as the proof of Lemma A.4 in [\ref{1-d seperable noncovex by ATY}].
\end{proof}

\section{Homogenization of Hamiltonian with Large Oscillation}\label{Proof of Large Oscillation}

In this section, the Hamiltonian is assumed to satisfy \textbf{(A1)-(A3)}, be constrained(c.f. Definition\ref{constrained Hamiltonian}) with index $(0,L)$ and has large oscillation(c.f. Definition \ref{oscillation}).

\subsection{Admissible decomposition and admissible functions} Recall \ref{m}, \ref{M} and denote
\begin{eqnarray*}
\underline{m}:=\essinf\limits\limits_{(x,\omega)}m(x,\omega) & \overline{M}:=\esssup\limits\limits_{(x,\omega)}M(x,\omega) & \mathcal{P}=(\underline{m},\overline{M})\bigcap [0,\infty)
\end{eqnarray*}

\begin{definition}
Fix any $\mu\in\mathcal{P}$ and $\omega\in\Omega$, a collection of disjoint finite intervals $\lbrace I_{i} \rbrace_{i\in\ZZ}$ is called a $(\mu,\omega)$ admissible decomposition of $\RR$ if the following (1),(2) and (3) hold.

(1)$I_{i}=(a_{i},a_{i+1}), \bigcup_{i\in\ZZ}[a_{i},a_{i+1}]=\RR$

(2)$\mu\in\lbrace m_{j}(a_{i},\omega),M_{j}(a_{i},\omega),m_{j}(a_{i},\omega),M_{j}(a_{i},\omega)|1\leq j\leq L \rbrace$.

(3)$\mu\notin \lbrace m_{j}(x,\omega),M_{j}(x,\omega)|1\leq j\leq L, x\in(a_{i},b_{i}) \rbrace$.

\end{definition}

\begin{remark}\label{remark about the decomposition}
Since $H(p,x,\omega)$ is constrained and has large oscillation, such $\lbrace I_{i} \rbrace_{i\in\ZZ}$ exists and is unique. By \textbf{(A1)}, for any $y\in\RR$, $\lbrace I_{i}-y \rbrace_{i\in\ZZ}$ is the $(\mu,\tau_{y}\omega)$ admissible decomposition of $\RR$.
\end{remark}

\begin{definition}
For fixed $\omega\in\Omega$ and $\mu\in\mathcal{P}$, let $\lbrace I_{i} \rbrace_{i\in\ZZ}$ be a $(\mu,\omega)$ admissible decomposition of $\RR$, then $f:\RR\rightarrow \RR$ is a $(\mu,\omega)$ admissible function if following (1),(2) and (3) hold.

(1)$0\leq f(x)\leq \max\lbrace p\geq 0| H(p,x,\omega)\leq \overline{M} \rbrace$.

(2)For each $i\in\ZZ$, $f(x)|_{I_{i}}=\psi_{j_{i},x}(\mu), \text{ for some }j_{i}\in\lbrace 1,2,\cdots,2L+1 \rbrace$.

(3)Any solution of $u^{\prime}=f(x)$ is a viscosity solution of the equation 
\begin{eqnarray}\label{equation for admissible set}
\begin{cases}
	H(u^{\prime}(x),x,\omega)=\mu\\
	u^{\prime}\geq 0
\end{cases}
 && x\in\RR
\end{eqnarray}

\end{definition}

\vspace{3mm}

\begin{definition}
For $\mu\geq 0$ and $\omega\in\Omega$, define
\begin{eqnarray*}
\mathcal{A}_{\mu}(\omega):=\begin{cases}
\lbrace \text{All } (\mu,\omega)\text{ admissible functions } \rbrace & \mu\in\mathcal{P}\\
\psi_{2L+1,x}(\mu) & \mu\leq \underline{m}\geq 0\\
\psi_{1,x}(\mu) & \mu\geq \overline{M}\\
\end{cases}
\end{eqnarray*}
\end{definition}

\begin{lemma}
$\mathcal{A}_{\mu}(\omega)\ne \emptyset$.
\end{lemma}

\begin{proof}
Fix $\omega\in\Omega$, by Lemma \ref{existence lemma}, there is a viscosity solution $u(x)$ of the equation (\ref{equation for admissible set}).

By Lemma \ref{decompsition lemma}, there is a strictly increasing sequence $\lbrace b_{i} \rbrace_{i\in\ZZ}$ such that
\begin{eqnarray*}
\lim\limits_{i\rightarrow\pm\infty}b_{i}=\pm\infty;\  u\in C^{1}((b_{i},b_{i+1})),  i\in\ZZ;\ u^{\prime}(x)|_{(b_{i},b_{i+1})}=\psi_{k_{i},x}(\mu),  k_{i}\in \lbrace 1,2,\cdots, 2L+1 \rbrace
\end{eqnarray*}

Let $\mu\in\mathcal{P}$ and $\lbrace I_{j} \rbrace_{j\in\ZZ}$ be the $(\mu,\omega)$ admissible decomposition of $\RR$. By refinement, we may assume that for $i\in\ZZ$, $(b_{i},b_{i+1})\subset I_{l_{i}}, l_{i}\in\ZZ$. For each $j\in\ZZ$, denote: $s(j)=\min\lbrace k_{i}|(b_{i},b_{i+1})\subset I_{j} \rbrace$. And define $f(x,\omega):=\psi_{s(j),x}(\mu), x\in I_{j}=(a_{j},a_{j+1})$. By Lemma \ref{combination lemma}, any solution to $u^{\prime}=f$ is a viscosity solution of the equation (\ref{equation for admissible set}).

Thus $f\in\mathcal{A}_{\mu}(\omega)$. If $\mu\notin \mathcal{P}$, it is clear that $\mathcal{A}_{\mu}(\omega)\ne \emptyset$.

\end{proof}

\begin{definition}
For each $\omega\in\Omega$ and $\mu\geq 0$, denote
\begin{eqnarray*}
\overline{f}_{\mu}(x,\omega)=\sup\lbrace f(x)|f\in \mathcal{A}_{\mu}(\omega) \rbrace &&
\underline{f}_{\mu}(x,\omega)=\inf\lbrace f(x)|f\in \mathcal{A}_{\mu}(\omega) \rbrace
\end{eqnarray*}
\end{definition}

\begin{lemma}\label{inf and sup of admissible functions}
$(1)$ For any $\mu\geq 0$ and $\omega\in\Omega$, $\overline{f}_{\mu}(x,\omega)$, $\underline{f}_{\mu}(x,\omega)\in\mathcal{A}_{\mu}(\omega).$

$(2)$$\overline{f}_{\mu}(x,\omega)\geq \underline{f}_{\mu}(x,\omega)$ and both of them are stationary.
\end{lemma}

\begin{proof}

(1)Fix any $\mu\geq 0$ and $\omega\in\Omega$.  For any point $x_{0}\in\RR$, since $H(p,x,\omega)$ is constrained with index $(0,L)$, there are $f_{r}\in\mathcal{A}_{\mu}(\omega)$, $\delta_{r}>0$ and $k_{r}\in\lbrace 1,2,\cdots, 2L+1 \rbrace$, such that
\begin{eqnarray*}
\overline{f}_{\mu}(x,\omega)|_{(x_{0},x_{0}+\delta_{r})}=f_{r}(x)|_{(x_{0},x_{0}+\delta_{r})}=\psi_{k_{r},x}(\mu)
\end{eqnarray*}

Similarly, there are $f_{l}\in\mathcal{A}_{\mu}(\omega)$, $\delta_{l}>0$ and $k_{r}\in\lbrace 1,2,\cdots, 2L+1 \rbrace$, such that
\begin{eqnarray*}
\overline{f}_{\mu}(x,\omega)|_{(x_{0}-\delta_{l},x_{0})}=f_{r}(x)|_{(x_{0}-\delta_{l},x_{0})}=\psi_{k_{l},x}(\mu)
\end{eqnarray*}

(i)If $k_{l}=k_{r}=k$. Then $\psi_{k,x}(\mu)$ is continuous on $(x_{0}-\delta_{l},x_{0}+\delta_{r})$. Since $H(\psi_{k,x}(\mu),x,\omega)=\mu$, any solution of $u^{\prime}=\psi_{k,x}(\mu)$ is the solution of the equation: $H(u^{\prime},x,\omega)=\mu, x\in(x_{0}-\delta_{l},x_{0}+\delta_{r})$

(ii)If $k_{l}<k_{r}$. It suffices to check any solution to $u^{\prime}=\overline{f}_{\mu}$ is a viscosity sub-solution at $x_{0}$. This follows from the following fact: $[\overline{f}(x_{0}^{+}),\overline{f}(x_{0}^{-})]=[f_{r}(x_{0}^{+}),f_{l}(x_{0}^{-})]\subset[f_{l}(x_{0}^{+}),f_{l}(x_{0}^{-})]$.

(iii)If $k_{l}>k_{r}$. It suffices to check any solution to $u^{\prime}=\overline{f}_{\mu}$ is a viscosity super-solution at $x_{0}$. This follows from the following fact: $[\overline{f}(x_{0}^{-}),\overline{f}(x_{0}^{+})]=[f_{l}(x_{0}^{-}),f_{r}(x_{0}^{+})]\subset[f_{r}(x_{0}^{-}),f_{r}(x_{0}^{+})]$.

So $\overline{f}_{\mu}(x,\omega)\in\mathcal{A}_{\mu}(\omega).$ Similarly, $\underline{f}_{\mu}(x,\omega)\in\mathcal{A}_{\mu}(\omega).$

(2)By definition, $\overline{f}_{\mu}(\cdot,\omega)\geq \underline{f}_{\mu}(\cdot,\omega)$. By Remark \ref{remark about the decomposition}, for any $y\in\RR$,
\begin{eqnarray*}
\overline{f}(x,\tau_{y}\omega)=\sup\lbrace f(x)|f(x)\in\mathcal{A}_{\mu}(\tau_{y}\omega)\rbrace
= \sup\lbrace f(x)|f(x-y)\in\mathcal{A}_{\mu}(\omega)\rbrace
= \overline{f}(x+y,\omega)
\end{eqnarray*}

Similarly, $\underline{f}(x,\tau_{y}\omega)=\underline{f}(x+y,\omega)$ for any $y\in\RR$.
\end{proof}

\subsection{Intermediate level set of the effective Hamiltonian}

\begin{lemma}\label{lemma of muu>inff M}
Let $H(p,x,\omega)$ satisfy \textbf{(A1)-(A3)} and be constrained with index $(0,L)$. If $\mu>\underline{M}$, then for a.e. $\omega\in\Omega$, the following is true: for any  $f(x)\in\mathcal{A}_{\mu}(\omega)$, there is a sequence of intervals $\lbrace J_{k} \rbrace_{k\in\ZZ}$ such that
\begin{eqnarray*}
J_{k}=(c_{k},c_{k+1}),\ \bigcup\limits_{k\in\ZZ}[c_{k},c_{k+1}]=\RR,\ \lim_{k\rightarrow\pm\infty}c_{k}=\pm\infty,\  f|_{J_{2k}}=\psi_{1,(x,\omega)}(\mu)
\end{eqnarray*}
\end{lemma}
\begin{proof}

By Lemma \ref{for a.e. omega}, for a.e. $\omega\in\Omega$, $\underline{M}=\essinf\limits\limits_{x\in\RR}M(x,\omega)$. Denote $\delta:=\mu-\underline{M}$ and $\epsilon:=\frac{\delta}{2}$. By ergodic theorem,
\begin{eqnarray*}
\lim_{L\rightarrow\pm\infty}\frac{1}{L}\int_{0}^{L}\CCChi_{\lbrace z,M(z,\omega)<\underline{M}+\epsilon \rbrace}(x)dx=\mathbf{E}\left[M(0,\cdot)<\underline{M}+\epsilon\right]>0 && \text{ a.e. } \omega\in\Omega
\end{eqnarray*}

So, almost surely, there is a sequence $x_{i}=x_{i}(\omega)$, such that $\lim\limits_{i\rightarrow\pm\infty}x_{i}=\pm\infty, M(x_{i},\omega)<\underline{M}+\epsilon$. By continuity of $M(x,\omega)$ in $x$, for each $i$, there is $\delta_{i}>0$, such that: $M(x,\omega)<\underline{M}+\epsilon,  x\in(x_{i}-\delta_{i},x_{i}+\delta_{i})$

Denote $c_{2k}:=x_{k}-\delta_{k},c_{2k+1}:=x_{k}+\delta_{k}, J_{k}:=(c_{k},c_{k+1})$. Then $f(x)|_{J_{2k}}=\psi_{1,(x,\omega)}(\mu)$ follows from the fact that: $H(f(x),x,\omega)=\mu>\underline{M}+\epsilon>M(x,\omega)|_{J_{2k}}, \text{ a.e. } \omega\in\Omega$.

\end{proof}

\begin{lemma}\label{lemma of muu>supp m}
Let $H(p,x,\omega)$ satisfy \textbf{(A1)-(A3)} and be constrained with index $(0,L)$. If $0\leq\mu<\overline{m}$, then for a.e. $\omega\in\Omega$, the following is true: for any $f(x)\in\mathcal{A}_{\mu}(\omega)$, there is a sequence of intervals $\lbrace J_{k} \rbrace_{k\in\ZZ}$ such that
\begin{eqnarray*}
J_{k}=(c_{k},c_{k+1}),\ \bigcup\limits_{k\in\ZZ}[c_{k},c_{k+1}]=\RR,\ \lim_{k\rightarrow\pm\infty}c_{k}=\pm\infty,\  f|_{J_{2k}}=\psi_{2L+1,x}(\mu)
\end{eqnarray*}
\end{lemma}

\begin{proof}
Similar argument as Lemma \ref{lemma of muu>inff M}.
\end{proof}

\begin{lemma}\label{level piece}
Let $H(p,x,\omega)$ satisfy \textbf{(A1)-(A3)} and be constrained with index $(0,L)$. Fix any $\mu\geq 0$ and $p\in[\int_{\Omega}\underline{f}_{\mu}(0,\omega)d\omega,\int_{\Omega}\overline{f}_{\mu}(0,\omega)d\omega]$, there is a stationary function $f(x,\omega):\RR\times\Omega\rightarrow \RR$ such that

(1)$p=\int_{\Omega}f(0,\omega)d\omega$.

(2)For a.e. $\omega\in\Omega$, any solution to $u^{\prime}=f(x,\omega)$ is a solution of the equation: $H(u^{\prime},x,\omega)=\mu$.
\end{lemma}

\begin{proof}
Suppose $\underline{u}^{\prime}(x,\omega)=\underline{f}_{\mu}(x,\omega)$ and $\overline{u}^{\prime}(x,\omega)=\overline{f}_{\mu}(x,\omega)$, Lemma \ref{inf and sup of admissible functions} implies $H(\underline{u}^{\prime},x,\omega)=\mu, H(\overline{u}^{\prime},x,\omega)=\mu$. Fix $\omega\in\Omega$, according to Lemma \ref{lemma of muu>inff M} and Lemma \ref{lemma of muu>supp m}, there exists a sequence of intervals $\lbrace I_{k} \rbrace_{k\in\ZZ}$, $I_{k}=(a_{k},a_{k+1})$, such that $\lim\limits_{k\rightarrow\pm\infty}a_{k}=\pm\infty$ and
\begin{eqnarray*}
\underline{f}_{\mu}(x,\omega)=\overline{f}_{\mu}(x,\omega), x\in I_{2k} &&
\underline{f}_{\mu}(x,\omega)\leq\overline{f}_{\mu}(x,\omega), x\in I_{2k+1}
\end{eqnarray*}

Denote
\begin{eqnarray*}
\underline{d}_{i}=\int_{a_{i}}^{a_{i+1}}\underline{f}_{\mu}(s,\omega)ds &&
\overline{d}_{i}=\int_{a_{i}}^{a_{i+1}}\overline{f}_{\mu}(s,\omega)ds
\end{eqnarray*}

For each $t\in[0,1]$, define $f_{t}:\RR\times\Omega\rightarrow \RR$ by
\begin{eqnarray*}
f_{t}(x,\omega):=\begin{cases}
\underline{f}_{\mu}(x,\omega)=\overline{f}_{\mu}(x,\omega) & x\in I_{2i}\\
\mathcal{F}_{I_{2i+1}}(\overline{f}_{\mu},\underline{f}_{\mu},t\overline{d}_{i}+(1-t)\underline{d}_{i}) & x\in I_{2i+1}\\
\end{cases}
\end{eqnarray*}

So $f_{t}(x,\omega)$ is stationary and
\begin{eqnarray*}
\int_{a_{0}}^{a_{i}}f_{t}(x,\omega)dx=\int_{a_{0}}^{a_{i}}t\overline{f}_{\mu}(x,\omega)+(1-t)\underline{f}_{\mu}(x,\omega)dx
\end{eqnarray*}

By \textbf{(A2)}, $\underline{f}_{\mu}$ and $\overline{f}_{\mu}$ are bounded. Then there is some constant $C>0$, such that
\begin{eqnarray*}
\frac{1}{\left|a_{i}-a_{0}\right|}\left|\int_{a_{0}}^{a_{i}}f_{t}(x,\omega)dx-\int_{a_{0}}^{a_{i}}f_{s}(x,\omega)dx\right|&=&\left|t-s\right|\left|\int_{a_{0}}^{a_{i}}\left(\overline{f}_{\mu}(s,\omega)-\underline{f}_{\mu}(s,\omega)\right)ds\right|\leq C\left|t-s\right|
\end{eqnarray*}

Thus 
\begin{eqnarray*}
\lim_{L\rightarrow\infty}\frac{1}{L}\left|\int_{0}^{L}f_{t}(x,\omega)dx-\int_{0}^{L}f_{s}(x,\omega)dx\right|\leq C \left|t-s\right|
\end{eqnarray*}

By ergodic theorem, for a.e. $\omega\in\Omega$,
\begin{eqnarray*}
\lim_{L\rightarrow \infty}\frac{1}{L}\int_{0}^{L}f_{t}(x,\omega)dx=\mathbf{E}\left[f_{t}(0,\omega)\right] &&
\lim_{L\rightarrow \infty}\frac{1}{L}\int_{0}^{L}f_{s}(x,\omega)dx=\mathbf{E}\left[f_{s}(0,\omega)\right]
\end{eqnarray*}

Then $\left|\mathbf{E}\left[f_{t}(0,\omega)\right]-\mathbf{E}\left[f_{s}(0,\omega)\right]\right|\leq C \left|t-s\right|$. So $\mathbf{E}\left[f_{t}(0,\omega)\right]$ is a continuous function of $t$, thus
\begin{equation*}
\bigcup_{t\in[0,1]} \mathbf{E}\left[f_{t}(0,\omega)\right] =\left[\int_{\Omega}\underline{f}_{\mu}(0,\omega)d\omega,\int_{\Omega}\overline{f}_{\mu}(0,\omega)d\omega\right]
\end{equation*}

So for any $p\in[\int_{\Omega}\underline{f}_{\mu}(0,\omega)d\omega,\int_{\Omega}\overline{f}_{\mu}(0,\omega)d\omega]$, there is $t=t(p)\in[0,1]$, s.t. $\mathbf{E}[f_{t}(0,\omega)]=p$. 

Let $u$ be the solution of $u^{\prime}=f_{t}(x,\omega)$. By Lemma \ref{homotopy between sol}, $u$ is a solution of $H(u^{\prime},x,\omega)=\mu$.

\end{proof}

\begin{lemma}
Let $H(p,x,\omega)$ satisfy \textbf{(A1)-(A3)} and be constrained with index $(0,L)$. Fix $\omega\in\Omega$, assume $\mu_{m}\rightarrow\mu$ and $f_{m}(x)\in\mathcal{A}_{\mu_{m}}(\omega)$. Then we have the following hold.
 
(1)If $\mu\in\mathcal{P}$, then $\limsup\limits_{m\rightarrow\infty}f_{m}(x)\in\mathcal{A}_{\mu}(\omega)$ and $
\liminf\limits_{m\rightarrow\infty}f_{m}(x)\in\mathcal{A}_{\mu}(\omega)$.

(2)If $\underline{m}\geq 0$ and $\mu\leq\underline{m}$, then except a countable set,
\begin{eqnarray*}
\limsup\limits_{m\rightarrow\infty}f_{m}(x)=\liminf\limits_{m\rightarrow\infty}f_{m}(x)=\psi_{2L+1,(x,\omega)}(\mu)
\end{eqnarray*}

(3)If $\mu\geq \overline{M}$, then except a countable set,
\begin{eqnarray*}
\limsup\limits_{m\rightarrow\infty}f_{m}(x)=\liminf\limits_{m\rightarrow\infty}f_{m}(x)=\psi_{1,(x,\omega)}(\mu)
\end{eqnarray*}

\end{lemma}

\begin{proof}
Only prove $f(x)=\limsup\limits_{m\rightarrow \infty}f_{m}(x)\in\mathcal{A}_{\mu}(\omega)$. The proof for $\liminf$ is similar.

\vspace{3mm}

(1)Let $\lbrace I_{i} \rbrace_{i\in\ZZ}$ be the $(\mu,\omega)$ admissible decomposition of $\RR$. Fix $k\in\ZZ$ and $\epsilon\ll 1$, there is $N\in\NN$, when $m>N$, $\mu_{m}\notin\lbrace m_{i}(x,\omega),M_{j}(x,\omega)|1\leq i,j\leq L, x\in(a_{k}+\epsilon,a_{k+1}-\epsilon)\cup (a_{k+1}+\epsilon,a_{k+2}-\epsilon) \rbrace$.

There are $l,\widetilde{l},q,\widetilde{q}\in\lbrace 1,2,\cdots,2L+1 \rbrace$, $\lbrace f_{l_{n}} \rbrace_{n\geq 1}$ and $\lbrace f_{q_{n}} \rbrace_{n\geq 1}$, such that
\begin{eqnarray*}
f_{l_{n}}(x)=\begin{cases}
\psi_{l,(x,\omega)}(\mu) & x\in(a_{k}+\frac{1}{n},a_{k+1}-\frac{1}{n})\\
\psi_{\widetilde{l},(x,\omega)}(\mu) & x\in(a_{k+1}+\frac{1}{n},a_{k+2}-\frac{1}{n})\\
\end{cases}
&& f_{q_{n}}(x)=\begin{cases}
	\psi_{\widetilde{q},(x,\omega)}(\mu) & x\in(a_{k}+\frac{1}{n},a_{k+1}-\frac{1}{n})\\
	\psi_{q,(x,\omega)}(\mu) & x\in(a_{k+1}+\frac{1}{n},a_{k+2}-\frac{1}{n})\\
\end{cases}
\end{eqnarray*}

\begin{eqnarray*}
f(x)|_{I_{k}}=\psi_{l,(x,\omega)}(\mu) && f(x)|_{I_{k+1}}=\psi_{q,(x,\omega)}(\mu)
\end{eqnarray*}

It suffices to show that the solution of $u^{\prime}=f$ is a viscosity solution of (\ref{equation for admissible set}) at $a_{k+1}$. 

Define $u_{l}\in W^{1,\infty}(a_{k},a_{k+2})$ and $u_{q}\in W^{1,\infty}(a_{k},a_{k+2})$ by solutions of
\begin{eqnarray*}
u_{l}^{\prime}(x)=\begin{cases}
\psi_{l,(x,\omega)}(\mu) & x\in I_{k}\\
\psi_{\widetilde{l},(x,\omega)}(\mu) & x\in I_{k+1}
\end{cases} &&
u_{q}^{\prime}(x)=\begin{cases}
\psi_{\widetilde{q},(x,\omega)}(\mu) & x\in I_{k}\\
\psi_{q,(x,\omega)}(\mu) & x\in I_{k+1}
\end{cases}
\end{eqnarray*}

By stability of viscosity solutions, $u_{l}$ and $u_{q}$ are both viscosity solutions to 
\begin{eqnarray*}
H(v^{\prime}(x),x,\omega)=\mu  && x\in(a_{k},a_{k+2})
\end{eqnarray*}

The jump of $f$ at $a_{k+1}$ is contained in the jump of $u_{l}^{\prime}$ or the jump of $u_{q}^{\prime}$ at $a_{k+1}$, so the solution of $u^{\prime}=f$ is a viscosity solution of (\ref{equation for admissible set}).

\vspace{3mm}

(2)Denote $A=\lbrace x\in\RR|\mu=m_{i}(x)\text{ for some }1\leq i\leq L \rbrace$. Since each of $m_{i}(x,\omega)$ has no cluster point, $A$ is countable. Since $\underline{m}\geq 0$ and $\mu\leq \underline{m}$, if $x\notin A$, then $\limsup\limits_{m\rightarrow\infty}f_{m}(x)=\liminf\limits_{m\rightarrow\infty}f_{m}(x)=\psi_{2L+1,(x,\omega)}(\mu)$.

\vspace{3mm}

(3)Denote $B=\lbrace x\in\RR|\mu=M_{j}(x)\text{ for some }1\leq j\leq L \rbrace$. Since each $M_{j}(x,\omega)$ has no cluster point, $B$ is countable. Since $\mu\geq \overline{M}$, if $x\notin B$, then $\limsup\limits_{m\rightarrow\infty}f_{m}(x)=\liminf\limits_{m\rightarrow\infty}f_{m}(x)=\psi_{1,(x,\omega)}(\mu)$.

\end{proof}

\begin{notation}
	For each $\mu\geq 0$, denote $\mathcal{I}_{\mu}=\left[\int_{\Omega}\underline{f}_{\mu}(0,\omega)d\omega,\int_{\Omega}\overline{f}_{\mu}(0,\omega)d\omega\right]$.
\end{notation}

\begin{remark}
	If $\mu\ne \nu$, then $\mathcal{I}_{\mu}\bigcap \mathcal{I}_{\nu}=\emptyset$. 
\end{remark}

\begin{lemma}
If $\lim\limits_{m\rightarrow \infty}\mu_{m}=\mu$, then
\begin{eqnarray*}
\int_{\Omega}\overline{f}_{\mu}(0,\omega)d\omega \geq \limsup_{m\rightarrow\infty}\int_{\Omega}\overline{f}_{\mu_{m}}(0,\omega)d\omega &&
\int_{\Omega}\underline{f}_{\mu}(0,\omega)d\omega \leq \liminf_{m\rightarrow\infty}\int_{\Omega}\underline{f}_{\mu_{m}}(0,\omega)d\omega
\end{eqnarray*}

Moreover, $\bigcup\limits_{\mu\geq 0}\mathcal{I}_{\mu}=[q_{0},\infty) $ with $ q_{0}=\int_{\Omega}\underline{f}_{0}(0,\omega)d\omega$.
\end{lemma}

\begin{proof}
Same as the proof of Lemma 3.8 in \ref{1-d seperable noncovex by ATY}.
\end{proof}

Denote $z_{l}(x,\omega):=\min\big\lbrace p\leq 0:H(q,x,\omega)\leq 0 \text{ on } [p,0] \big\rbrace$.

\subsection{Extreme level set of effective Hamiltonian}

\begin{lemma}\label{minimum level negative}
Let $H(p,x,\omega)$ satisfy \textbf{(A1)-(A3)} and be constrained with index $(0,L)$. For any $p\in[\mathbf{E}[z_{l}(0,\omega)],\mathbf{E}[\underline{f}_{0}(0,\omega)]]$, there is a stationary function $f(x,\omega)$ such that $p=\mathbf{E}[f(0,\omega)]$ and any solution to $u^{\prime}=f$ is a viscosity sub-solution of $H(u^{\prime},x,\omega)=0, x\in\RR$.
\end{lemma}

\begin{proof}
Since $H(p,x,\omega)$ is constrained with index $(0,L)$, $\overline{m}=\esssup\limits\limits_{(x,\omega)\in\RR\times\Omega}m(x,\omega)>0$. And by similar arguments in Lemma \ref{lemma of muu>inff M}, then: for a.e. $\omega\in\Omega$, there is $\lbrace b_{i} \rbrace_{i\in\ZZ}$ such that
\begin{eqnarray*}
\lim_{i\rightarrow\pm\infty}b_{i}=\pm\infty, & m(x,\omega)|_{(b_{2i},b_{2i+1})}\in\left[\frac{3\overline{m}}{4},\overline{m}\right], &
m(x,\omega)|_{(b_{2i+1},b_{2i+2})}\leq \frac{3}{4}\overline{m}
\end{eqnarray*}

For each $i\in\ZZ$, denote $\underline{r}_{i}=\int_{b_{i}}^{b_{i+1}}z_{l}(x,\omega)d\omega$ and $
\overline{r}_{i}=\int_{b_{i}}^{b_{i+1}}\underline{f}_{0}(x,\omega)dx$. Fix $t\in(0,1)$, define a stationary function $f_{t}(x,\omega):\RR\times\Omega\rightarrow \RR$ by the following procedure.

\vspace{3mm}

STEP 1: Modification on $(b_{2i},b_{2i+1})$. Denote
\begin{eqnarray*}
f_{l,t}(x,\omega)=\begin{cases}
(1-t)\underline{f}_{0}(x,\omega)+tz_{l}(x,\omega) & x\in\bigcup\limits_{i\in\ZZ}(b_{2i},b_{2i+1})\\
z_{l}(x,\omega) & x\in\bigcup\limits_{i\in\ZZ}[b_{2i+1},b_{2i+2}]\\
\end{cases}
\end{eqnarray*}
\begin{eqnarray*}
f_{r,t}(x,\omega)=\begin{cases}
(1-t)\underline{f}_{0}(x,\omega)+tz_{l}(x,\omega) & x\in\bigcup\limits_{i\in\ZZ}(b_{2i},b_{2i+1})\\
\underline{f}_{0}(x,\omega) & x\in\bigcup\limits_{i\in\ZZ}[b_{2i+1},b_{2i+2}]\\
\end{cases}
\end{eqnarray*}

Since $H(p,x,\omega)$ is convex in $p$ on $(z_{l}(x,\omega),\underline{f}_{0}(x,\omega))$ for all $x\in(b_{2i},b_{2i+1})$, if $u$ is the a solution of the equation $u^{\prime}=f_{l,t}$ or $u^{\prime}=f_{r,t}$, then in viscosity sense, we have $H(u^{\prime}(x,\omega),x,\omega)\leq 0, x\in\RR$

\vspace{3mm}

STEP 2: Modification on $[b_{2i+1},b_{2i+2}]$. Define
\begin{eqnarray*}
f_{t}:=\begin{cases}
\mathcal{F}_{I_{2i+1}}(\underline{f}_{0},z_{l}(x,\omega),(1-t)\overline{r}_{i}+t\underline{r}_{i}) & x\in [b_{2i+1},b_{2i+2}]\\
f_{l,t}(x,\omega)=f_{r,t}(x,\omega) & x\in (b_{2i},b_{2i+1})\\
\end{cases}
\end{eqnarray*}

By Lemma \ref{homotopy between sol}, if $u^{\prime}=f_{t}$, then in viscosity sense, we have $H(u^{\prime}(x,\omega),x,\omega)\leq 0, x\in\RR$.

By similar arguments as Lemma \ref{level piece}, there is some constant $C>0$, such that
\begin{eqnarray*}
\frac{1}{|b_{i}-b_{0}|}\left|\int_{b_{0}}^{b_{i}}f_{t}(x,\omega)dx-\int_{b_{0}}^{b_{i}}f_{s}(x,\omega)dx\right|\leq C|t-s| 
\end{eqnarray*} 
\begin{eqnarray*}
\lim_{L\rightarrow+\infty}\frac{1}{L}\left|\int_{0}^{L}f_{t}(x,\omega)dx-\int_{0}^{L}f_{s}(x,\omega)dx\right|\leq C|t-s|
\end{eqnarray*}

So $\mathbf{E}[f_{t}(0,\omega)]$ is a continuous function. Since $\mathbf{E}[f_{0}(0,\omega)]=\mathbf{E}[\underline{f}_{0}(0,\omega)], \mathbf{E}[f_{1}(0,\omega)]=\mathbf{E}[z_{l}(0,\omega)]$, it concludes that $\bigcup\limits_{t\in[0,1]} \mathbf{E}[f_{t}(0,\omega)]=[\mathbf{E}[z_{l}(0,\omega)],\mathbf{E}[\underline{f}_{0}(0,\omega)]]$. So for any $p\in[\mathbf{E}[z_{l}(0,\omega)],\mathbf{E}[\underline{f}_{0}(0,\omega)]]$, there is $t=t(p)$, such that $p=\mathbf{E}[f_{t}(0,\omega)]$, then any solution of $u^{\prime}=f_{t}(x,\omega)$ is a viscosity sub-solution of $H(v^{\prime},x,\omega)=0,x\in\RR$.

\end{proof}

\begin{lemma}\label{minimum level positive}
Let $H(p,x,\omega)$ satisfy \textbf{(A1)-(A3)} and be constrained with index $(0,L)$. Then for a.e. $\omega\in\Omega$, we have: fix $p\in\RR$, let $v_{\lambda}(\cdot,\omega)\in W^{1,\infty}(\RR)$ be the unique viscosity solution of the equation: $\lambda v_{\lambda}+H(p+v_{\lambda}^{\prime},x,\omega)=0, x\in\RR$, then $\liminf\limits_{\lambda\rightarrow 0}-\lambda v_{\lambda}(x,\omega)\geq 0$.
\end{lemma}

\begin{proof}
By assumption, $\essinf\limits\limits_{(x,\omega)\in\RR\times\Omega} H(0,x,\omega)< 0$, for each $(x,\omega)\in\RR\times\Omega$, denote
\begin{eqnarray*}
V(x,\omega):= \min\lbrace H(0,x,\omega), m(x,\omega)\rbrace
\end{eqnarray*}

Then $V(x,\omega)\leq 0$ and it is a bounded continuous stationary function. Then
\begin{eqnarray*}
H_{+}(p,x,\omega):=H(p,x,\omega)-V(x,\omega)\geq 0
\end{eqnarray*}

For a.e. $\omega\in\Omega$, by similar argument as Lemma \ref{lemma of muu>inff M}: for any $\delta>0$, there are 
\begin{eqnarray*}
I_{i}=(a_{i},a_{i+1})&\lim\limits_{i\rightarrow\pm \infty}a_{i}=\pm \infty &
-\delta\leq V(x,\omega)\leq 0, \ x\in(a_{2i},a_{2i+1})
\end{eqnarray*}

Then $\liminf\limits_{\substack{\lambda\rightarrow 0 \\x\in(a_{2i},a_{2i+1})}}-\lambda v_{\lambda}(x,\omega)\geq -\delta$.

On the other hand, for each $\omega\in\Omega$, there is a sequence $\lambda_{n}\rightarrow 0$ and a constant $C\in\RR$, such that
\begin{eqnarray*}
-\lambda_{n}v_{\lambda_{n}}(x,\omega)\rightarrow C && \text{ locally uniformly in }\RR
\end{eqnarray*}

So $C\geq-\delta$. Since $\delta>0$ can be arbitrary, $C\geq 0$. Thus $\liminf\limits_{\lambda\rightarrow 0}-\lambda v_{\lambda}(x,\omega)\geq 0$.

\end{proof}

\begin{remark}
By Lemma \ref{minimum level negative} and Lemma \ref{minimum level positive}, for any $p\in[\mathbf{E}[z_{l}(0,\omega)],\mathbf{E}[\underline{f}_{0}(0,\omega)]]$, $H(p,x,\omega)$ is regularly homogenizable and $\overline{H}(p)=0$.
\end{remark}

\begin{lemma}
For $p\in(-\infty,\mathbf{E}[z_{l}(0,\omega)])$, $H(p,x,\omega)$ is regularly homogenizable.
\end{lemma}

\begin{proof}
For each $\mu\geq 0$, denote $p_{\mu}=\mathbf{E}[\Psi_{(0,\omega)}(\mu)]$, let $v(x,\omega)$ be the solution of the equation
\begin{eqnarray*}
v^{\prime}(x,\omega)=\Psi_{(x,\omega)}(\mu)-p_{\mu}
\end{eqnarray*}

Then $v$ is a sub-linear solution of $H(p+v^{\prime},x,\omega)=\mu, x\in\RR$. The lemma follows from the fact that
\begin{eqnarray*}
(-\infty,\mathbf{E}[z_{l}(0,\omega)])=\bigcup_{\mu>0} \lbrace p_{\mu} \rbrace
\end{eqnarray*}
\end{proof}

\begin{remark}
From the construction of the effective Hamiltonian $\overline{H}(p)$, in the case of large oscillation, $\overline{H}(p)$ is coercive, continuous and level-set convex.
\end{remark}

\textbf{Acknowledgement:} The author would like to thank his advisor Yifeng Yu for his helpful guidance and generous support.

\bibliographystyle{amsplain}

\end{document}